\documentclass[10pt]{article}

\usepackage{bm}
\usepackage{bbm}
\usepackage{fullpage}  % for narrow margins

\usepackage{url}
\usepackage{algpseudocode}
\usepackage{multirow} %//multirow for format of table
\usepackage{hhline}  % for \hhline in tables
\usepackage{verbatim}
\usepackage{graphicx}
\usepackage{amsmath,amssymb,amsfonts,graphicx,amsthm,mathtools,nicefrac}%},mathrsfs}
\usepackage{lscape}
\usepackage{color}
\usepackage{authblk}
\usepackage{subfigure}
\usepackage[ruled,vlined]{algorithm2e}
\DeclareMathOperator*{\argmax}{arg\,max}
\DeclareMathOperator*{\argmin}{arg\,min}
\usepackage{bbding}
\usepackage{indentfirst}
\usepackage{graphicx}
\newcommand\sbullet[1][.5]{\mathbin{\vcenter{\hbox{\scalebox{#1}{$\bullet$}}}}}

\usepackage{multirow}
\allowdisplaybreaks
\newcommand\numberthis{\addtocounter{equation}{1}\tag{\theequation}}
\newtheorem{definition}{Definition}[section]
\newtheorem{theorem}{Theorem}[section]
\newtheorem{assumption}{Assumption}[section]
\newtheorem{remark}{Remark}[section]

\numberwithin{equation}{section}

\newtheorem{lemma}[theorem]{Lemma}
\newtheorem{proposition}[theorem]{Proposition}
\newtheorem{corollary}[theorem]{Corollary}

\DeclareMathOperator{\Var}{Var}

     \def\R{\mathbb{R}}

\def\calA{{\cal  A}} 
\def\calB{{\cal  B}} 
\def\calC{{\cal  C}}

\def\calO{{\cal  O}}

\def\calS{{\cal  S}}

\usepackage{bbm}

\usepackage{amsmath,mathrsfs,amsfonts,amssymb} 
\usepackage[dvipsnames]{xcolor}

\newcommand{\kw}[1]{{#1}}

\newcommand{\ljc}[1]{{#1}}

\begin{document}

\title{On the Convergence of Projected Policy Gradient for Any Constant Step Sizes}
\author[1]{Jiacai Liu}
\author[1]{Wenye Li}
\author[2]{Dachao Lin}
\author[1]{Ke Wei}
\author[3]{Zhihua Zhang}
\affil[1]{School of Data Science, Fudan University, Shanghai, China}
\affil[2]{ Huawei Technologies Shanghai R\&D Center, Shanghai China}
\affil[3]{School of Mathematical Sciences, Peking University, Beijing, China}%\vspace{.15cm}}
\date{\today}

%%%%%%
\maketitle
%%%%%%
\begin{abstract}
Projected policy gradient (PPG) is a basic policy optimization method in reinforcement learning. Given access to exact policy  evaluations, previous studies have established the sublinear convergence of PPG for sufficiently small step sizes based on the smoothness and the gradient domination properties of the value function. However, as the step size  goes to infinity, PPG reduces to the classic policy iteration method, which suggests the convergence of PPG even for large step sizes. In this paper, we fill this gap and show that PPG admits a sublinear convergence {\em for any constant step sizes}. Due to the existence of the state-wise visitation measure in the expression of policy gradient, the existing optimization-based analysis framework for a preconditioned version of PPG (i.e., projected Q-ascent) is not applicable, to the best of our knowledge. Instead, we proceed the proof by computing the state-wise improvement lower bound of PPG based on its inherent structure. In addition, the finite iteration convergence of PPG for any constant step size is further established, which is also new.
\\

\noindent
\textbf{Keywords.} Projected policy gradient, sublinear convergence, finite iteration convergence, policy optimization, policy iteration
\end{abstract}

\section{Introduction}
Reinforcement learning (RL) %is a type of machine learning technique for solving sequential decision problems and has recently achieved great success in many areas, such as games \cite{mnih2015humanlevel,SilverHuangEtAl16nature,Berner2019Dota2W,StarCraft}, robotics \cite{robot1,robot2,robot3} and various other real applications \cite{Agarwal2016MakingCD,chen-2019-topk,chipdesign}. 
\kw{is essentially about how to make efficient sequential decisions to achieve a long term goal. It has received intensive investigations  both from theoretical
and algorithmic aspects due to its recent success in many areas}, such as games \cite{mnih2015humanlevel,SilverHuangEtAl16nature,Berner2019Dota2W,StarCraft}, robotics \cite{robot1,robot2,robot3} and various other real applications \cite{Agarwal2016MakingCD,chen-2019-topk,chipdesign}. 
Typically, RL can be modeled as a discounted Markov decision process (MDP) represented by a tuple $\mathcal{M} \left(\calS ,\calA ,P,r,\gamma ,\mu \right)$, where $\calS$ is the state space, $\calA$ denotes the action space, $P(s'|s,a)$ is the transition probability or density from state $s$ to state $s'$ under action $a$, $r: \calS \times \calA \times \calS\rightarrow \mathbb{R}$ is the reward function, $\gamma \in [0,1)$ is the discounted factor and $\mu$ is the probability distribution of the initial state $s_0$.  In this paper, we  focus the tabular setting where $\calS$ and $\calA$ are finite, i.e., $|\calS|<\infty$ and $|\calA|<\infty$. %The key concepts in RL are state values (value functions) and state-action values (Q functions) defined by
Let $\Delta(\calA)$ be the probability simplex over the set $\calA$, %(resp. $\Delta(\calS)$) 
defined as
\begin{align}
\Delta \left( \calA \right) =\left\{ \theta \in \mathbb{R} ^{|\calA |}: \theta _i\ge 0,\sum_{i=1}^{|\calA |}{\theta _i}=1 \right\}.
\end{align}
The set of admissible stationary policies  \kw{(i.e., the direct or simplex parameterization of policies)} is given by
\begin{align}
\Pi :=\left\{ \pi =\left( \pi_s \right) _{s\in \mathcal{S}}\,\,|\,\, \pi_s\in \Delta \left( \mathcal{A} \right) \,\,\text{for all } s\in \mathcal{S} \right\}, 
\label{Policy Space}
\end{align}
where $\pi_s := \pi(\cdot | s) \in \mathbb{R}^{|\calA|}$ and $\pi \in \mathbb{R}^{|\calS| \times |\calA|}$.

Given a policy $\pi \in \Pi$, the state value function at  $s\in\calS$ is defined as %and the state-action value function at $(s,a)\in\calS\times \calA$ are defined as
\begin{align}
V^{\pi}\left( s \right) &:=\mathbb{E} \left\{ \sum_{t=0}^{\infty}{\gamma ^tr\left( s_t,a_t,s_{t+1} \right)}|s_0=s,\pi \right\},
\label{V} 
\end{align}
while the state-action value function at $(s,a)\in\calS\times \calA$ are defined as
\begin{align}
Q^{\pi}\left( s,a \right) &:=\mathbb{E} \left\{ \sum_{t=0}^{\infty}{\gamma ^tr\left( s_t,a_t,s_{t+1}\right)}|s_0=s,a_0=a,\pi \right\}.
\label{Q}
\end{align}
 % where $s\in \calS, a\in \calA$ and $\pi:\calS \rightarrow \Delta(\calA)$ is the (stationary) policy. 
%Under the case of policy sequence $\{ \pi^k \}$, we omit notation $\pi$ for both value function and state-action value function, e.g. $V^{k}(s) := V^{\pi}(s)|_{\pi=\pi^k}$. 
Overall, the goal of RL is to find a policy that maximizes the weighted average of the state values under the initial distribution $\mu$, namely to solve
\begin{align}
\underset{\pi \in \Pi}{\max}\,V^{\pi}\left( \mu \right). \label{objective}
\end{align}
Here $
V^{\pi}\left( \rho \right) :=\mathbb{E} _{s\sim \rho}\left[ V^{\pi}\left( s \right) \right]$ for any $\rho \in \Delta(\calS)$. %and we denote $\pi^*$ as the optimal policy, $V^*$ and $Q^*$ as the value function and state-action value function of $\pi^*$. 

{Policy optimization refers to a family of effective methods in reinforcement learning. In this paper, we focus on}
projected policy gradient (PPG) \kw{ which is likely to be the most direct optimization method for solving} \eqref{objective}. Given an initial policy $\pi_0 \in \Pi$, PPG generates a policy sequence $\{\pi^k\}$ for $k = 1,2,3,...$ as follows:
\begin{align*}
\pi ^{k+1}& = \underset{\pi \in \Pi}{\mathrm{arg}\max}\left\{ \eta _k\left< \nabla _{\pi}V^{\pi}\left( \mu \right) \left| _{\pi =\pi^k} \right. ,\pi -\pi ^k \right> -\frac{1}{2}\left\| \pi -\pi ^k \right\| _{2}^{2} \right\} ,
\\
\,\,     &=\underset{\pi \in \Pi}{\mathrm{arg}\max}\left\{ \sum_{s\in \mathcal{S}}{\left(\eta _k\left< \nabla _{\pi _s}V^{\pi}\left( \mu \right) \left| _{\pi =\pi ^k} \right. ,\pi _s-\pi _{s}^{k} \right> -\frac{1}{2}\left\| \pi _s-\pi _{s}^{k} \right\| _{2}^{2}\right)} \right\},
\end{align*}
or state-wisely,
\begin{align*}
 \pi _{s}^{k+1}=\underset{\pi \in \Pi}{\mathrm{arg}\max}\left\{ \eta _k\left< \nabla _{\pi _s}V^{\pi}\left( \mu \right) \left| _{\pi =\pi ^k} \right. ,\pi _s-\pi _{s}^{k} \right> -\frac{1}{2}\left\| \pi _s-\pi _{s}^{k} \right\| _{2}^{2} \right\}.\numberthis
\label{simplex_pg_proximal_view}
\end{align*}
  According to the policy gradient theorem \cite{pg}, $$
\nabla _{\pi _s}V^{\pi}\left( \mu \right) =\frac{d_{\mu}^{\pi}\left( s \right)}{1-\gamma}Q^{\pi}(s, \cdot)
$$where $d_{\mu}^{\pi}$ is the state visitation probability defined as  
\begin{align*}
d_{\mu}^{\pi}\left( s \right) :=\left( 1-\gamma \right) \mathbb{E} \left\{ \sum_{t=0}^{\infty}{\gamma ^t \mathbbm{1}_{[s_t=s]}}|s_0\sim \mu ,\pi \right\}.\numberthis\label{eq:visitation}
\end{align*}
Thus,  \kw{PPG can be written explicitly in the following form}:
\begin{align}
\mbox{(PPG)} \quad\pi _{s}^{k+1}=\mathrm{Proj}_{\Delta \left( \mathcal{A} \right)}\left( \pi_s^k + \frac{\eta_k d^k_\mu(s)}{1-\gamma}Q^k(s,\cdot) \right), \quad \forall s\in \mathcal{S},
\label{simplex PG}
\end{align}
\kw{where $d^k_\mu$ and $Q^k(s,\cdot)$ are short for $d^{\pi_k}_\mu$ and $Q^{\pi_k}(s,\cdot)$, respectively, and $\text{Proj}_{\Delta(\mathcal{A})}$ denotes the projection onto $\Delta(\mathcal{A})$, i.e., $\mathrm{Proj}_{\Delta \left( \mathcal{A} \right)}\left( v \right) =\underset{p\in \Delta \left( \mathcal{A} \right)}{\mathrm{arg}\min}\,\,\left\| p-v \right\| _{2}^{2}$.}
 \kw{Note that removing the visitation measure $d^k_\mu(s)$  in the PPG update leads to   the projected Q-ascent (PQA) method,}
\begin{align*}
\mbox{(PQA)}\quad \pi _{s}^{k+1}=\mathrm{Proj}_{\Delta \left( \mathcal{A} \right)}\left( \pi _{s}^{k}+\eta _kQ^{k}(s,\cdot) \right),\quad\forall s\in \mathcal{S}.
\numberthis
\label{projected Q-descent}
\end{align*}
PQA is indeed a special case of policy mirror ascent methods (e.g. \cite{Geist_Scherrer_Pietquin_2019,Shani_Efroni_Mannor_2019,Lan_2021,Xiao_2022,Cen_Cheng_Chen_Wei_Chi_2022,Zhan_Cen_Huang_Chen_Lee_Chi_2021,Li_Zhao_Lan_2022,Johnson_Pike-Burke_Rebeschini_2023}) where  the Bregman distance is the squared $\ell_2$-distance and it can also be seen as a preconditioned version of PPG. 

%%%%%%%%%%%%%%%%%
% 主要改的部分之一, 
{
%%%%
\subsection{Motivation and contributions}
The convergence of PPG has been investigated in \cite{Agarwal_Kakade_Lee_Mahajan_2019,  bhabdari2024or, Zhang_Koppel_Bedi_Szepesvari_Wang_2020,Xiao_2022},  given the access to exact policy evaluations. More precisely, it is shown in \cite{Agarwal_Kakade_Lee_Mahajan_2019, bhabdari2024or} that PPG converges to a global optimum at an $O(1/\sqrt{k})$ sublinear rate, which has been improved to $O(1/k)$ subsequently in \cite{Zhang_Koppel_Bedi_Szepesvari_Wang_2020,Xiao_2022}. The analyses in these works  all utilize the smoothness property of the value function, and thus require the step size to be smaller than $1/L$, where $L = \frac{2\gamma \left| \mathcal{A} \right|}{\left( 1-\gamma \right) ^3}$ is the smoothness coefficient of the value function \cite{Agarwal_Kakade_Lee_Mahajan_2019}. However, as $\eta_k$ goes to infinity, it is easy to see from \eqref{simplex_pg_proximal_view} that 
PPG approaches the classic policy iteration (PI) method. Therefore, due to the convergence of policy iteration, it is natural to expect PPG also converges for large step sizes. 

Motivated by the above observation, we extend the convergence studies of PPG {to any constant step sizes} in this paper. The main contributions of this paper are summarized  as follows:
\begin{itemize}
    \item The $O(1/k)$ sublinear convergence of PPG has been established for any constant step sizes, see Theorem~\ref{theorem:sublinear of PPG}. In order to break the step size limitation hidden in the existing optimization analysis framework, we adopt a different route and leverage the more explicit form of the projection onto the probability simplex to derive a state-wise improvement lower bound for PPG. {\em It is worth noting that, due to the existence of the visitation measure in PPG, the analysis for PQA within the framework of policy mirror ascent in \textup{\cite{Xiao_2022, Lan_2021}} is not applicable for PPG, to the best of our knowledge. In fact, the sublinear convergence results of PPG \textup{(}only for sufficiently small step sizes\textup{)} and PQA \textup{(}for any constant step sizes\textup{)} have been established  separately based on different techniques in \textup{\cite{Xiao_2022}}}.
    \item We further show that PPG indeed terminates after a finite number of iterations. The finite iteration convergence of PQA for any constant step size can also be obtained in a similar way. Note that, as a special  case of a general result  in \cite{Li_Zhao_Lan_2022}, the homotopic PQA can be shown to converge in a finite number of iterations. However, this does not imply the finite convergence of PQA for any constant step sizes and the homotopic PQA basically requires an exponentially increasing step size to converge. As a by-product, we present a new dimension-free bound for the finite iteration convergence of PI and VI, which does not explicitly depend on $|\mathcal{S}|$ and $|\mathcal{A}|$.
%    \item Inspired by \cite{Johnson_Pike-Burke_Rebeschini_2023},  the global $\gamma$-rate linear convergence  of PPG using non-adaptive geometrically increasing step sizes is established. We also show that both PPG and PQA are equivalent to PI when the step size $\eta_k$ is larger than a threshold that can be calculated from the current policy $\pi^k$.
\end{itemize}
In addition to the main contributions, we also give a brief discussion on the $\gamma$-rate linear convergence  of PPG using non-adaptive geometrically increasing step sizes, as well as the equivalence of PPG and PQA to policy iteration when the step size $\eta_k$ is larger than a threshold that can be calculated from the current policy $\pi^k$.
The existing convergence results and our new results on PPG (as well as on PQA for completeness)  are summarized in Table~\ref{tab:results}.

\begin{table}[ht!]
{\caption{Convergence results for PPG and PQA.}}\label{tab:results}
\renewcommand{\arraystretch}{1.2}
\centering
\begin{tabular}{c|l|l}
\hline
& {\hspace{1cm}Existing results} & \hspace{3cm} {New results}\\
\hline
\multirow{ 4}{*}{PPG} & & $\sbullet[.75]$ Sublinear convergence for any constant $\eta_k$\\
&{$\sbullet[.75]$ Sublinear convergence } & $\sbullet[.75]$ Finite iteration convergence  for any \\
&  {\color{white}$\sbullet[.75]$}  when $\eta_k\leq 1/L$ \cite{Agarwal_Kakade_Lee_Mahajan_2019, bhabdari2024or, Zhang_Koppel_Bedi_Szepesvari_Wang_2020,Xiao_2022}& {\color{white}$\sbullet[.75]$}  constant $\eta_k$\\
&& $\sbullet[.75]$ $\gamma$-rate linear convergence  for  \\
& & {\color{white}$\sbullet[.75]$} geometrically increasing step sizes\\
\hline
\multirow{6}{*}{PQA} &$\sbullet[.75]$ Sublinear convergence   &\\ 
&{\color{white}$\sbullet[.75]$} for any constant $\eta_k$ \cite{Lan_2021,Xiao_2022} &\\
&$\sbullet[.75]$  Finite iteration convergence   & $\sbullet[.75]$ Finite iteration convergence for any\\
&{\color{white}$\sbullet[.75]$}  for homotopic PQA \cite{Li_Zhao_Lan_2022} & {\color{white}$\sbullet[.75]$} constant $\eta_k$\\
& $\sbullet[.75]$ $\gamma$-rate / linear convergence    & \\
& {\color{white}$\sbullet[.75]$} for geometrically increasing& \\
&  {\color{white}$\sbullet[.75]$} step sizes \cite{Xiao_2022, Johnson_Pike-Burke_Rebeschini_2023} & \\
\hline
\end{tabular}
\end{table}
%%%%
}
%%%%%%%%%%%%%%%%%

\subsection{Notation and assumptions}
Recalling the definitions of the state value function (\ref{V}) and the state-action value function (\ref{Q}), the advantage  function of a policy $\pi$ is defined as
\begin{align*}
A^{\pi}\left( s,a \right) :=Q^{\pi}\left( s,a \right) -V^{\pi}\left( s \right).
\end{align*}
%The notations for $A$ can completely follow the case of $Q$, e.g. $A^*(s,a) := A^\pi(s,a) |_{\pi=\pi^*}$. 
\kw{It is evident that} $A^\pi(s,a)$ measures how well a single action is  compared with the average state value. \kw{Moreover, we  use $V^*(s)$, $Q^*(s,a)$ and $A^*(s,a)$ to denote the corresponding value functions associated with the optimal policy $\pi^*$, and use $V^k(s)$, $Q^k(s,a)$ and $A^k(s,a)$ to denoted the corresponding value functions associated with the policy output by the algorithm in the $k$-th iteration. In the sequel we often use the shorthand notation for ease of exposition, for example,
\begin{align*}
\pi_{s,a}:=\pi(a|s),\quad \pi_s:=\pi(\cdot|s),\quad Q^\pi_{s,a}:=Q^\pi(s,a),\quad\mbox{and}\quad Q^\pi_s:=Q^\pi(s,\cdot).
\end{align*}}

Given a state $s \in \calS$,    the set of optimal actions $\calA_s^\ast$ at state $s$ is defined as,
\begin{align*}
    \calA_s^\ast = \arg\max_{a\in\calA} Q^\ast(s,a) = \arg\max_{a\in\calA} A^\ast(s,a).
\end{align*}
 Given a policy $\pi \in \Pi$, a state $s \in \calS$ and a set $B \subset \calA$,
define
$$
\pi _s\left( B \right) =\sum_{a\in B}{\pi_s \left( a\right)}
$$
as the probability of $\pi_s$ on $B$ and \kw{denote by
$b_s^\pi$  the probability on non-optimal actions},
$$
b_{s}^{\pi}=\pi _s\left( \mathcal{A} \setminus \mathcal{A} _{s}^{*} \right).
$$
\kw{When $b_s^\pi$ is small for any $s\in\mathcal{S}$, it is natural to expect that $\pi$ will be close to be optimal. Thus, $b_s^\pi$ is a very essential optimality measure of a policy.} 
The set of $\pi$-optimal actions at state $s$, \kw{denoted $\mathcal{A} _{s}^{\pi}$}, is defined as $$
\mathcal{A} _{s}^{\pi}=\underset{a\in \mathcal{A}}{\mathrm{arg}\max}\,\,A^{\pi}\left( s,a \right),
$$
\kw{with $\mathcal{A} _{s}^{k}$ being the abbreviation of  $\mathcal{A} _{s}^{\pi^k}$}.
 \kw{The following quantity is quite central in the finite iteration convergence analysis}, which has also appeared in previous works, see for example \cite{Mei_Xiao_Szepesvari_Schuurmans_2020,Khodadadian_Jhunjhunwala_Varma_Maguluri_2021}.

\begin{definition} The optimal advantage function gap $\Delta$ is defined as follows:
\begin{align}
\Delta :=\underset{s\in \tilde{S},a \notin \mathcal{A} _{s}^{*}}{\min}\left| A^{*}(s,a) \right|,
\label{optimal gap}
\end{align}
where $\widetilde{S}=\{s \in \calS:\calA^*_s \ne \calA\}$  denotes the set of states that have non-optimal actions. 
\end{definition}
Without
loss of generality, we assume $\tilde{\calS} \neq \emptyset$. It is trivial that $\Delta>0$ since $A^*(s,a)<0$ holds for all non-optimal actions. Additionally, we will make the following two standard assumptions about the reward and the initial state distribution.
%Without the loss of generality We only consider the case of $\tilde{\calS} \neq \emptyset$. Finally, besides the simplification introduced before (i.e. $Q^k$, $Q^*$ and so on), we denote
%\begin{align*}
%    \pi_{s,a} := \pi(a|s), \;\; Q^\pi_{s,a} := Q^\pi(s,a), \;\; A^\pi_{s,a} := A^\pi(s,a)
%\end{align*}
%in our analysis and proofs. For function $V^\pi$, $Q^\pi$ and $A^\pi$ we remove subscripts to denote the corresponding vectors, e.g. for $V^\pi$ and $Q^\pi$
%\begin{align*}
%    V^\pi := \left(V^\pi(s)\right)_{s \in \calS} \in \mathbb{R}^{|\calS|}, \;\; Q^\pi_{s} := \left(Q^\pi(s,a)\right)_{a \in \calA} \in \mathbb{R}^{|\calA|}, \;\; Q^\pi := \left(Q^\pi(s,a)\right)_{(s,a) \in \calS \times \calA} \in \mathbb{R}^{|\calS|\cdot|\calA|},
%\end{align*}
%and similarly for advantage function $A^\pi$.

%dditionally, we set some assumptions for our object DMDP:

\begin{assumption}[Bounded reward] 
    $r(s,a,s') \in [0,1],~\forall s, s' \in \calS,~ a \in \calA.$
    \label{bounded reward}
\end{assumption}

\begin{assumption}[Traversal initial distribution] $\tilde{\mu}:=\underset{s\in \calS}{\min}\,\mu \left( s \right) >0.$
    \label{traversal initial distribution}
\end{assumption}

Recall that $d_\mu^\pi$ defined \eqref{eq:visitation} is the state visitation measure following policy $\pi$. We use $d^*_\mu$ to the state visitation measure following the optimal policy $\pi^*$ and use $d^k_\mu$ to denote the visitation measure following the policy output by the algorithm in the $k$-th iteration. For $\pi \in \Pi$, $\mu \in \Delta(\calS)$ and $s \in \calS$, \kw{it follows immediately from Assumption~\ref{traversal initial distribution} that} 
    \begin{align*}
        d^\pi_\mu(s) \geq (1-\gamma)\tilde{\mu}.\numberthis \label{bound of d}
    \end{align*}

\subsection{Organization of the paper}
\kw{The rest of the paper is outlined as follows.}
In Section~\ref{Preliminary}, some preliminary results are provided which be used in our later analysis. The sublinear convergence of PPG with any constant step size is discussed in Section~\ref{Sublinear Convergence of PPG}, \kw{followed by} the finite convergence in Section ~\ref{Finite time convergence results}. The dimension-free bound for the finite iteration convergence of PI and VI is also presented in Section~\ref{Finite time convergence results}.
 In Section \ref{Linear convergence and the equivalence to PI}, we present the results of linear convergence and equivalence to PI under different step size selection rules.

\section{Preliminaries}
\label{Preliminary}

%Here we provide some useful preliminary lemmas. The first part is under the general DMDP setting, including the relationships between value function $V$, state-action value function $Q$ and Advantage function $A$. The second part introduces the properties of Eucliden projection to simplex.

\subsection{Useful lemmas }
As we assume the reward function $r$ is bounded in Assumption~\ref{bounded reward}, all the value functions are bounded as they are discounted summations of rewards.

\begin{lemma}
    For any policy $\pi \in \Pi$ and $(s,a) \in \calS \times \calA$,
    \begin{align*}
        V^\pi(s) \in \left[0, \frac{1}{1-\gamma}\right], \;\; Q^\pi(s,a) \in \left[0, \frac{1}{1-\gamma} \right], \;\; A^\pi(s,a) \in \left[-\frac{1}{1-\gamma}, \frac{1}{1-\gamma} \right].
    \end{align*}
    \label{bounds of V Q A}
\end{lemma}

By leveraging the structure property of the MDP, we further have the lemma below.

\begin{lemma} 
\label{lemma: advantage_Lemma}
For any policy $\pi$ ,
\begin{itemize}
    \item  $
\left\| Q^*-Q^{\pi} \right\| _{\infty}\le \gamma \left\| V^*-V^{\pi} \right\| _{\infty}
$
    \item $\left\| A^{*}-A^{\pi} \right\| _{\infty}\le \left\| V^*-V^{\pi} \right\| _{\infty}$
    \item $
\left\| V^*-V^{\pi} \right\| _{\infty}\le \frac{V^*\left( \rho \right) -V^{\pi}\left( \rho \right)}{\tilde{\rho}}$ for any $\rho \in \Delta(\calS)$ such that $
\tilde{\rho}:=\underset{s \in \calS}{\min}\,\rho \left( s \right) > 0$.
\end{itemize}

\end{lemma}
\begin{proof}
Recalling the definition of state-action value function \eqref{Q}, one has
$$
\left| Q^{\pi}\left( s,a \right) -Q^*\left( s,a \right) \right|=\gamma \,\Big| \mathbb{E} _{s^{\prime}\sim P\left( \cdot |s,a \right)}\left[ V^{\pi}\left( s^{\prime} \right) -V^*\left( s^{\prime} \right) \right] \Big|\le \gamma \left\| V^{\pi}-V^* \right\| _{\infty}.
$$
For the advantage function, one has
$$
A^{\pi}\left( s,a \right) -A^*\left( s,a \right) =\left( V^*\left( s \right) -V^{\pi}\left( s \right) \right) -\left( Q^*\left( s,a \right) -Q^{\pi}\left( s,a \right) \right).
$$
On the one hand,
$$
A^{\pi}\left( s,a \right) -A^*\left( s,a \right) \le V^*\left( s \right) -V^{\pi}\left( s \right) \le \left\| V^*-V^{\pi} \right\| _{\infty}.
$$
On the other hand,
$$
A^*\left( s,a \right) -A^{\pi}\left( s,a \right) \le Q^*\left( s,a \right) -Q^{\pi}\left( s,a \right) \le \gamma \left\| V^*-V^{\pi} \right\| _{\infty}.
$$
Thus $\left\| A^{\pi}-A^* \right\|_\infty \le \left\| V^*-V^{\pi} \right\| _{\infty}$. For the bound on $\left\| V^*-V^{\pi} \right\| _{\infty}$, a direct computation yields 
$$
\left\| V^*-V^{\pi} \right\| _{\infty} \le \sum_s{\frac{\rho \left( s \right)}{\rho \left( s \right)}\left( V^*\left( s \right) -V^\pi\left( s \right) \right)}\le \frac{V^*\left( \rho \right) -V^\pi\left( \rho \right)}{\tilde{\rho}},
$$
\kw{which concludes the proof.}
\end{proof}

The performance difference lemma below is a fundamental lemma in the analysis of RL algorithms \ljc{(e.g. \cite{Agarwal_Kakade_Lee_Mahajan_2019,Mei_Xiao_Szepesvari_Schuurmans_2020,Khodadadian_Jhunjhunwala_Varma_Maguluri_2021,hadamard_pg,Xiao_2022})}. It characterizes the difference between the value functions of two arbitrary policies can be represented as the weighted average of the advantages.
\begin{lemma}[Performance Difference Lemma \cite{kakade2002approximately}]
\label{performanceDifferenceLemma}
	For any two policies $\pi_1, \pi_2$, and any $\rho \in \Delta(\calS)$, one has
	\begin{align*}
		V^{\pi_1}(\rho) - V^{\pi_2}(\rho) =\frac{1}{1-\gamma} \mathbb{E}_{s\sim d_{\rho}^{\pi_1}}\left[ \mathbb{E}_{a \sim \pi_1(\cdot|s)}\left[A^{\pi_2}(s,a)\right] \right].
	\end{align*}
\end{lemma}

Recall that $b^\pi_s$ denotes the probability on non-optimal actions which can be viewed as an essential measure for the optimality of a policy. %By intuition, smaller $b^\pi_s$ over $s\in\calS$ indicates $\pi$ is closer to $\pi^*$. 
The following two lemmas establish the relation between $b_s^\pi$ and the mismatch $V^{\ast}(\rho )-V^{\pi}(\rho )$.

\begin{lemma} [\protect{\cite[Theorem~3.1]{Khodadadian_Jhunjhunwala_Varma_Maguluri_2021}}]
For any policy $\pi \in \Pi$ and $\rho \in \Delta(\calS)$,
$$
V^{\ast}(\rho )-V^{\pi}(\rho )\le \frac{1}{\left( 1-\gamma \right) ^2}\cdot \mathbb{E} _{s\sim d_{\rho}^{\pi}}\left[ b_{s}^{\pi} \right] .
$$
\label{lemma: V error le b}
\end{lemma}
\begin{lemma}
For any policy  $\pi \in \Pi$  and $\rho \in \Delta(\calS)$,
$$
\mathbb{E} _{s\sim \rho}\left[ b_{s}^{\pi} \right] \le \frac{V^*\left( \rho \right) -V^{\pi}\left( \rho \right)}{\Delta}.
$$
\label{lemma: b le V error}
\end{lemma} 
\begin{proof}
According to the performance difference lemma,
\begin{align*}
&\phantom{==}V^{\ast}(\rho )-V^{\pi}(\rho )=-\left( V^{\pi}\left( \rho \right) -V^*\left( \rho \right) \right) 
\\
\,\,&=\frac{1}{1-\gamma}\sum_s{d_{\rho}^{\pi}\left( s \right) \sum_a{\pi \left( a|s \right) \cdot \left( -A^*\left( s,a \right) \right)}}
\\
&=\frac{1}{1-\gamma}\sum_{s\in \tilde{\mathcal{S}}}{d_{\rho}^{\pi}\left( s \right) \sum_{a\notin \mathcal{A} _{s}^{*}}{\pi \left( a|s \right) \cdot \left| A^*\left( s,a \right) \right|}}
\\
\,\,&\ge \frac{1}{1-\gamma}\sum_{s\in \tilde{\mathcal{S}}}{d_{\rho}^{\pi}\left( s \right) \sum_{a\notin \mathcal{A} _{s}^{*}}{\pi \left( a|s \right) \cdot \Delta}}
\\
&=\frac{\Delta}{1-\gamma}\sum_{s\in \tilde{\mathcal{S}}}{d_{\rho}^{\pi}\left( s \right) b_{s}^{\pi}}
\ge \Delta\cdot \mathbb{E} _{s\sim \rho}\left[ b_{s}^{\pi} \right].
\end{align*}
The proof is complete \kw{after rearrangement}.
\end{proof}

%At last, we give a lower bound for the state visitation probability $d^\pi_\mu$.

%\input{theorems/section 2/lemma: bound of d}

\subsection{Basic facts about projection onto probability simplex}
Recall that Euclidian projection onto the probability simplex is defined as
\begin{align*}
    \mathrm{Proj}_{\Delta(\calA)}(p) = \underset{y\in\Delta(\calA)}{\arg\min} \| y-p \|^2.
\end{align*}
This projection has an explicit expression, presented in the following lemma.

\begin{lemma} For arbitrary vector $
p=\left( p_a \right) _{a\in \mathcal{A}}\in \mathbb{R} ^{\left| \mathcal{A} \right|}$,
$$
\mathrm{Proj}_{\Delta \left( \mathcal{A} \right)}\left( p \right) =\left( p+\lambda \mathbf{1}\right) _+
$$
where  $\lambda$ is a constant such that  $\sum_{a \in \mathcal{A}}{\left( p_a+\lambda \right) _+}=1$.
\label{proj}
\end{lemma}
\begin{proof} This fact can be obtained by studying the KKT condition of the projection problem, see for example  \cite{simplex_projection} for details.
\end{proof}

\begin{remark} It's trivial to see that the projection onto probability simplex has a shift-invariant property. That is, $
\mathrm{Proj}_{\Delta \left( \mathcal{A} \right)}\left( p \right) =\mathrm{Proj}_{\Delta \left( \mathcal{A} \right)}\left( p+c \mathbf{1} \right) $ holds for arbitrary constant $c \in \mathbb{R}$.
\kw{Therefore, PPG and PQA can also be expressed in terms of advantages functions. For example, we have the following alternative expression for PPG: 
\begin{align*}
\pi _{s}^{k+1}=\mathrm{Proj}_{\Delta \left( \mathcal{A} \right)}\left( \pi_s^k + \frac{\eta_k d^k_\mu(s)}{1-\gamma}A^k(s,\cdot) \right), \quad \forall s\in \mathcal{S}.
\end{align*}}
\label{proj_property_1}
\end{remark}

%Lemma \ref{proj} means that the Euclidean projection is just to add a constant shift in the original vector $p$ and truncate the negative values with zero. Besides the shift-invariant property in remark \ref{proj_property_1}, there is another useful property called gap property which is used throughout our analysis.
\kw{Lemma \ref{proj} implies that the projection onto the probability simplex can be computed by first translating the vector with an offset, followed by truncating those negative values to zeros. Moreover, the next lemma provides a characterization on the support of the projection, which will be used frequently in our analysis.}

\begin{lemma}[Gap property]
Let $\mathcal{B}$ and $\mathcal{C}$ be two disjoint non-empty sets such that $\mathcal{A} =\mathcal{B}\cup \mathcal{C}$. Given an arbitrary vector $
p=\left( p_a \right) _{a\in \mathcal{A}}\in \mathbb{R} ^{\left| \mathcal{A} \right|}$, let $y=\mathrm{Proj}_{\Delta \left( \mathcal{A} \right)}\left( p \right)$. Then
$$
\forall a^\prime\in \mathcal{C},~ y_{a^\prime}=0 \Leftrightarrow \,\,\sum_{a\in \mathcal{B}}{\left( p_a-\underset{a^{\prime}\in \mathcal{C}}{\max}\,\,p_{a^{\prime}} \right) _+}\ge 1.
$$
\label{proj_property_2}    
\end{lemma}

%\begin{remark} This lemma indicates that if the values of the vector $p$ in the index set $\mathcal{C}$ are generally smaller than the values in index set $\mathcal{B}$ and the cumulative gap satisfies $\sum_{a\in \mathcal{B}}{\left( p_a-\underset{a^\prime\in \mathcal{C}}{\max} \,\,p_{a^\prime} \right)}$, then the index set $\mathcal{C}$ is excluded from the support set of the projected vector $y$. We will use this lemma to demonstrate that both PPG and PQD converge once the value error enters into the small-$\varepsilon$ regima thus both methods converge in a finite time.
%\end{remark}

Roughly speaking, this lemma indicates that if the entries of  $p$ in the index set $\mathcal{C}$ are generally smaller than \kw{those} in the index set $\mathcal{B}$ and the cumulative gap is larger than 1, then the index set $\mathcal{C}$ will be excluded from the support set of the projection $y$. The proof of this lemma is essentially contained in the argument for Theorem~1 in \cite{simplex_projection}. \kw{To keep the
presentation self-contained, we give a short proof below.}

\begin{proof} [Proof of Lemma~\ref{proj_property_2}]
First note that $\forall a \in \calA$, 
$$
y_a = 0 \,\stackrel{(a)}{\Longleftrightarrow}\, \lambda \leq -p_a \,\stackrel{(b)}{\Longleftrightarrow}\, \sum_{a^\prime \in \calA} \left( p_{a^\prime} - p_a \right)_+ \geq 1,
$$
where $(a)$ follows from Lemma~\ref{proj} and $(b)$ is due to $\sum_{a\in\calA} (p_a + \lambda)_+ = 1$ and the monotonicity of $(\cdot)_+$. Thus we have
\begin{align*}
 y_{a^\prime}=0, ~~\forall a^\prime\in \mathcal{C} \Leftrightarrow \,\,\underset{a^\prime\in \mathcal{C}}{\min}\sum_{a\in \mathcal{A}}{\left( p_a-\,p_{a^{\prime}} \right) _+}\ge 1&\Leftrightarrow \sum_{a\in \mathcal{A}}{\left( p_a-\underset{a^\prime\in \mathcal{C}}{\max}\,\, p_{a^{\prime}} \right) _+}\ge 1 \\
 &\Leftrightarrow \sum_{a\in \mathcal{B}}{\left( p_a-\underset{a^\prime\in \mathcal{C}}{\max}\,\,p_{a^{\prime}} \right) _+}\ge 1,
\end{align*}
which completes the proof.
\end{proof}

%Recall that both of the PPG and PQD policy update (E.q. \ref{simplex PG} and \ref{projected Q-descent}) contain projection. By lemmas above we can reformulate both two update into one expression.

\kw{Based on Lemma~\ref{proj}, we can now rewrite the one-step update of PPG in \eqref{simplex PG}
and PQA in \eqref{projected Q-descent} into a unified framework with an explicit expression for the projection. This is the prototype update that will be mainly analysed. }
 Given an input policy $\pi\in \Pi$ and step size $\eta>0$, the new policy $\pi^+$  is generated via 
\begin{align*}
\mbox{(Prototype Update)}\quad \quad \pi^+_{s,a}(\eta_s)=\left( \pi_{s,a} + \eta_s A_{s,a}^\pi +\lambda _s \right) _+,\quad\forall s\in \mathcal{S} ,a\in \mathcal{A}, \numberthis\label{policy update of PPG and PQD}
\end{align*}
where $\lambda_s$ is a constant such that $\sum_{a} {\pi}^+_{s,a} = 1$. Note that, given a constant step size $\eta$, $\eta _{s} = \frac{\eta d_{\mu}^{\pi}\left( s \right)}{1-\gamma}$ for PPG while $\eta_s=\eta$ for PQA.
%\input{theorems/section 2/corollary: policy update of PPG and PQD}

%We present our analysis with respect to the update form above (E.q. \eqref{policy update of PPG and PQD}) throughout the discussion later, and give results for PPG and PQD individually in the form of corollary.

\subsection{Basic properties of prototype update}

Before presenting the sublinear convergence of PPG, we give a brief discussion on the basic properties of the prototype update.  The lemma below presents all the possibilities for  the support set of the new policy $\pi^+_{s}(\eta)$ (the subscript $s$  in $\eta_s$ will be omitted in this section for simplicity).

\begin{lemma}
    Consider the prototype update in \eqref{policy update of PPG and PQD}. Denote by $\calB_s(\eta)$  the support set of $\pi^+_s$:
    \begin{align*}
        \calB_s(\eta) := \left\{ a: \pi^+_{s,a}(\eta) > 0 \right\}.
    \end{align*}
    Then for any $ \eta > 0$,  $\calB_s(\eta)$ \kw{admits} one of the following three forms:
    \begin{enumerate}
        \item $\calB_s(\eta) \subsetneqq \calA^\pi_s$,
        \item $\calB_s(\eta) = \calA^\pi_s$,
        \item $\calB_s(\eta) = \calA^\pi_s \cup \calC_s(\eta)$, where $\calC_s(\eta) \subseteq \calA \setminus \calA_s^\pi$ is not empty.
    \end{enumerate}
    \label{three cases of support set}
\end{lemma}

\begin{proof}  Without loss of generality, assume $\mathcal{A}^\pi_s\neq \mathcal{A}$. Then
\kw{it suffices to} show that if $\calB_s(\eta)$ contains an action $a^\prime \notin \calA^\pi_s$, all $\pi$-optimal actions are included in $\calB_s(\eta)$. \ljc{
Given any $
\hat{a} \in\mathcal{A}_s^\pi$, define  $$
\hat{\mathcal{A}}_s=\left\{ a \in \calA: \pi _{s,a}+\eta A_{s,a}^{\pi} \ge \pi _{s,\hat{a}}+\eta A_{s,\hat{a}}^{\pi} \right\}.$$ Next we will show that  
\begin{align*}
I:=&\sum_{a\in \mathcal{A}}{\left( \pi _{s,a}+\eta A_{s,a}^{\pi}-\left( \pi _{s,\hat{a}}+\eta A_{s,\hat{a}}^{\pi} \right) \right) _+}\\
=&\sum_{a\in \hat{\mathcal{A}}_s}{\left( \pi _{s,a}+\eta A_{s,a}^{\pi}-\left( \pi _{s,\hat{a}}+\eta A_{s,\hat{a}}^{\pi} \right) \right)}
\\
\,\,                                             =&\sum_{a\in \hat{\mathcal{A}}_s \cap \mathcal{A} _{s}^{\pi}}{\left( \pi _{s,a}-\pi _{s,\hat{a}} \right)}+\sum_{a\in \hat{\mathcal{A}}_s\setminus \mathcal{A} _{s}^{\pi}}{\left( \pi _{s,a}-\pi _{s,\hat{a}}+\eta \left( A_{s,a}^{\pi}-A_{s,\hat{a}}^{\pi} \right) \right)}
<1,
\end{align*}
from which the claim follows immediately using Lemma~\ref{proj_property_2}.

If $\tilde{\mathcal{A}}_s\setminus \mathcal{A} _{s}^{\pi}\ne \emptyset$, then 
\begin{align*}
\sum_{a\in \hat{\mathcal{A}}_s\setminus \mathcal{A} _{s}^{\pi}}{\left( \pi _{s,a}-\pi _{s,\hat{a}}+\eta \left( A_{s,a}^{\pi}-A_{s,\hat{a}}^{\pi} \right) \right)} <\sum_{a\in \hat{\mathcal{A}}_s\setminus \mathcal{A} _{s}^{\pi}}{\left(\pi _{s,a}-\pi _{s,\hat{a}} \right)}, 
\end{align*}
since $A_{s,a}^{\pi}<A_{s,\hat{a}}^{\pi}$ for $a\in \hat{\mathcal{A}}_s\setminus \mathcal{A} _{s}^{\pi}$. Consequently,
\begin{align*}
    I <\sum_{a\in \hat{\mathcal{A}}_s \cap \mathcal{A} _{s}^{\pi}}{\left( \pi _{s,a}-\pi _{s,\hat{a}} \right)}+\sum_{a\in \hat{\mathcal{A}}_s\setminus \mathcal{A} _{s}^{\pi}}{\left(\pi _{s,a}-\pi _{s,\hat{a}} \right)}\leq 1.
\end{align*}
On the other hand, if $\hat{\mathcal{A}}_s\setminus \mathcal{A} _{s}^{\pi}= \emptyset$ , one has $
\hat{\mathcal{A}}_s\subset\mathcal{A} _{s}^{\pi}$. In this case, 
$$
I=\sum_{a\in \mathcal{A} _{s}^{\pi}}{\left( \pi _{s,a}-\pi _{s,\hat{a}} \right)}\le \pi _s\left( \mathcal{A} _{s}^{\pi} \right) <1,
$$
where the last inequality is due to the fact that  $\calB_s(\eta)$ contains an action $a^\prime \notin \calA^\pi_s$.
}
\end{proof}

The \kw{last lemma} implies that at least one $\pi$-optimal action is included in the support set $\calB_s(\eta)$. In addition, it is not hard to verify that  when step size $\eta$ goes to infinity every $\pi$-suboptimal actions will be excluded from the support set of $\pi^+_s$, \kw{which suggests that  $\calB_s(\eta)$ might shrink} as $\eta$ increases. The following lemma confirms that this observation is indeed true. 

\begin{lemma}
    For $\eta_1 > \eta_2 > 0$ we have $$\calB_s(\eta_1) \subseteq \calB_s(\eta_2).$$
    \label{inclusion relationship of support set}
\end{lemma}

\begin{proof}
Since the relation holds trivially when $\calB_s(\eta_2)=\mathcal{A}$, we only consider the case $\calB_s(\eta_2)\neq\mathcal{A}$.
 First the application of Lemma \ref{proj_property_2} yields that
    \begin{align}
        \sum_{a\in\calB_s(\eta)} \left[ \pi_{s,a} + \eta A^\pi_{s,a} - \max_{a^\prime \not\in\calB_s(\eta)} (\pi_{s,a^\prime} + \eta A^\pi_{s,a^\prime})  \right]_+ \geq 1
        \label{eq: gap1}
    \end{align}
    and
    \begin{align}
        a^\prime \not\in \calB_s(\eta) \Longleftrightarrow \sum_{a\neq a^\prime} \left[ \pi_{s,a} + \eta A^\pi_{s,a} - (\pi_{s,a^\prime} + \eta A^\pi_{s,a^\prime})  \right]_+ \geq 1.
        \label{eq: gap2}
    \end{align}
 If $\calB_s(\eta_2) \subseteq \calA^\pi_s$ (i.e. the first two cases in Lemma \ref{three cases of support set}), then for any $a^\prime \not\in \calB_s(\eta_2)$ we have
    \begin{align*}
        &\sum_{a\neq a^\prime} \left[ \pi_{s,a} + \eta_1 A^\pi_{s,a} - (\pi_{s,a^\prime} + \eta_1 A^\pi_{s,a^\prime}) \right]_+\\
        &\geq \sum_{a\in\calB_s(\eta_2)} \left[ \pi_{s,a} + \eta_1 A^\pi_{s,a} - (\pi_{s,a^\prime} + \eta_1 A^\pi_{s,a^\prime}) \right]_+ \\
        &\stackrel{(a)}{\geq} \sum_{a\in\calB_s(\eta_2)} \left[ \pi_{s,a} + \eta_2 A^\pi_{s,a} - (\pi_{s,a^\prime} + \eta_2 A^\pi_{s,a^\prime}) \right]_+ \geq 1,
    \end{align*}
    where $(a)$ is due to the fact $(A^\pi_{s,a} - A^\pi_{s,a^\prime}) \geq 0$ for $\forall a \in \calB_s(\eta_2) \subseteq \calA^\pi_s$. This implies that $a^\prime \not\in \calB_s(\eta_1)$, and thus $\calB_s(\eta_1) \subseteq \calB_s(\eta_2)$.

    For the case that $\calB_s(\eta_2) = \calA_s \cup \calC_s(\eta_2)$, fixing $a^\prime \not\in\calB_s(\eta_2)$, it follows from \eqref{eq: gap1} that
    \begin{align*}
        &\sum_{a\in\calB_s(\eta_2)} \left[ \pi_{s,a} + \eta_2 A^\pi_{s,a} -  (\pi_{s,a^\prime} + \eta_2 A^\pi_{s,a^\prime})  \right]_+  \\
        \geq &\sum_{a\in\calB_s(\eta_2)} \left[ \pi_{s,a} + \eta_2 A^\pi_{s,a} - \max_{a^\prime \not\in\calB_s(\eta_2)} (\pi_{s,a^\prime} + \eta_2 A^\pi_{s,a^\prime})  \right]_+ \geq 1.
    \end{align*}
    %by e.q. \eqref{eq: gap1}. 
    Furthermore, since $a^\prime \not\in \calB_s(\eta_2)$ it is trivial to see that $\forall a\in\calB_s(\eta_2)$,
    $
\pi_{s,a} + \eta_2 A^\pi_{s,a} > \pi_{s,a^\prime} + \eta_2 A^\pi_{s,a^\prime}.
    $
    Therefore,
    \begin{align*}
        &\sum_{a\in\calB_s(\eta_2)} \left[ \pi_{s,a} + \eta_2 A^\pi_{s,a} -  (\pi_{s,a^\prime} + \eta_2 A^\pi_{s,a^\prime})  \right]_+  \\
        &=
        \sum_{a\in\calB_s(\eta_2)} \left[ \pi_{s,a} + \eta_2 A^\pi_{s,a} -  (\pi_{s,a^\prime} + \eta_2 A^\pi_{s,a^\prime})  \right]   \\
        &= \sum_{a\in\calB_s(\eta_2)} \left[ \pi_{s,a} -  \pi_{s,a^\prime} + \eta_2( A^\pi_{s,a} - A^\pi_{s,a^\prime})  \right] \\
        &\geq 1,
    \end{align*}
    which yields 
    $
        \sum_{a\in\calB_s(\eta_2)} (A^\pi_{s,a} - A^\pi_{s,a^\prime}) \geq 0,
    $
    as $\eta_2 > 0$. \kw{Consequently},
    \begin{align*}
        &\sum_{a\in\calB_s(\eta_2)} \left[ \pi_{s,a} - \pi_{s,a^\prime} + \eta_1 (A^\pi_{s,a} - A^\pi_{s,a^\prime})  \right]_+ \hspace{-.1cm}- \hspace{-.2cm}\sum_{a\in\calB_s(\eta_2)} \left[ \pi_{s,a} - \pi_{s,a^\prime} + \eta_2 (A^\pi_{s,a} - A^\pi_{s,a^\prime})  \right]_+ \\
        &=\sum_{a\in\calB_s(\eta_2)} \left[ \pi_{s,a} - \pi_{s,a^\prime} + \eta_1 (A^\pi_{s,a} - A^\pi_{s,a^\prime})  \right]_+ \hspace{-.1cm}- \hspace{-.2cm}\sum_{a\in\calB_s(\eta_2)} \left[ \pi_{s,a} - \pi_{s,a^\prime} + \eta_2 (A^\pi_{s,a} - A^\pi_{s,a^\prime})  \right] \\
        &\geq\sum_{a\in\calB_s(\eta_2)} \left[ \pi_{s,a} - \pi_{s,a^\prime} + \eta_1 (A^\pi_{s,a} - A^\pi_{s,a^\prime})  \right] - \hspace{-.2cm}\sum_{a\in\calB_s(\eta_2)} \left[ \pi_{s,a} - \pi_{s,a^\prime} + \eta_2 (A^\pi_{s,a} - A^\pi_{s,a^\prime})  \right] \\
        &=(\eta_1 - \eta_2)\sum_{a\in\calB_s(\eta_2)}(A^\pi_{s,a} - A^\pi_{s,a^\prime}) > 0,
    \end{align*}
    yielding 
    \begin{align*}
        &\sum_{a\neq a^\prime} \left[ \pi_{s,a} - \pi_{s,a^\prime} + \eta_1 (A^\pi_{s,a} - A^\pi_{s,a^\prime})  \right]_+ \geq \sum_{a\in\calB_s(\eta_2)} \left[ \pi_{s,a} - \pi_{s,a^\prime} + \eta_1 (A^\pi_{s,a} - A^\pi_{s,a^\prime})  \right]_+\\
        &\geq \sum_{a\in\calB_s(\eta_2)} \left[ \pi_{s,a} - \pi_{s,a^\prime} + \eta_2 (A^\pi_{s,a} - A^\pi_{s,a^\prime})  \right]_+ \geq 1.
    \end{align*}
    Together with \eqref{eq: gap2} we have $a^\prime \not\in\calB_s(\eta_1)$, which implies $\calB_s(\eta_1) \subseteq \calB_s(\eta_2)$.
\end{proof}

%Finally we give a useful bound for $A^\pi_{s,a}$ when the support set $\calB_s(\eta)$ is given. Recall the intuition to e.q. \ref{policy update of PPG and PQD} that actions with larger advantages $A^\pi_{s,a}$ will more likely be chosen into the support set $\calB_s(\eta)$. To be more concretely, the lemma below provides a lower bound for the advantage of the action selected into $\calB_s(\eta)$.
The following lemma shows that $A^{\pi}_{s,a}$ should be sufficiently large in order to be included in the support set of $\pi_s^+$, which is reasonable.

\begin{lemma}  Consider the prototype update in \eqref{policy update of PPG and PQD}. We have
$$
A_{s,a}^{\pi}\ge \underset{{\tilde{a}}\in \mathcal{A}}{\max}\,\,A_{s,{\tilde{a}}}^{\pi}-\frac{2\pi _s\left( \mathcal{A} \setminus \mathcal{A} _{s}^{\pi} \right)}{\eta},\quad \forall a\in \calB_s(\eta).
$$
\label{lemma: advantage lower bound for support actions}
\end{lemma}

\begin{proof} It suffices to consider the third case  in Lemma \ref{three cases of support set}. According to Lemma \ref{proj_property_2}, for any action $a \in \calB_s(\eta) \setminus \calA_s^\pi$, we have
\begin{align*}
1  &> \sum_{a^\prime\in \mathcal{A}}{\left( \pi _{s,a^\prime}+\eta A_{s,a^\prime}^{\pi}-\pi _{s,a}-\eta A_{s,a}^{\pi} \right) _+}
\\
\,\,&\ge \sum_{a^\prime\in \mathcal{A} _{s}^{\pi}}{\left( \pi _{s,a^\prime}+\eta A_{s,a^\prime}^{\pi}-\pi _{s,a}-\eta A_{s,a}^{\pi} \right) _+}
\\
\,\,&=\sum_{a^\prime\in \mathcal{A} _{s}^{\pi}}{\left( \pi _{s,a^\prime}-\pi _{s,a}+\eta\left( \underset{\tilde{a}\in \mathcal{A}}{\max}A_{s,\tilde{a}}^{\pi}-A_{s,a}^{\pi} \right) \right) _+}
\\
\,\,&\ge \sum_{a^\prime\in \mathcal{A} _{s}^{\pi}}{\left( \pi _{s,a^\prime}-\pi _{s,a}+\eta\left( \underset{\tilde{a}\in \mathcal{A}}{\max}A_{s,\tilde{a}}^{\pi}-A_{s,a}^{\pi} \right) \right)}
\\
\,\,&= \pi _s\left( \mathcal{A} _{s}^{\pi} \right) +\left| \mathcal{A} _{s}^{\pi} \right|\left( \eta\left( \underset{\tilde{a}\in \mathcal{A}}{\max}A_{s,\tilde{a}}^{\pi}-A_{s,a}^{\pi} \right) -\pi _{s,a} \right).
\end{align*}
It follows that 
\[
\eta\left( \underset{\tilde{a}\in \mathcal{A}}{\max}A_{s,\tilde{a}}^{\pi}-A_{s,a}^{\pi} \right) \le \pi _{s,a}+\frac{1-\pi _s\left( \mathcal{A} _{s}^{\pi} \right)}{\left| \mathcal{A} _{s}^{\pi} \right|}\le \left( 1\hspace{-0.1cm}+\hspace{-0.1cm}\frac{1}{\left| \mathcal{A} _{s}^{\pi} \right|} \right) \pi _s\left( \mathcal{A} \setminus \mathcal{A} _{s}^{\pi} \right) \le 2\pi _s\left( \mathcal{A} \setminus \mathcal{A} _{s}^{\pi} \right) .
\]
The proof is complete after rearrangement.
\end{proof}

\section{Sublinear convergence of PPG for any constant step size}
\label{Sublinear Convergence of PPG}
As already mentioned, the sublinear convergence of PQA for any constant step size has already been developed in \cite{Xiao_2022}. Even though PPG and PQA are overall similar to each other, to the best of our knowledge, the  technique  for the sublinear convergence analysis of PQA cannot be used to establish the sublinear convergence of PPG for any constant step size due to the existence of the visitation measure. Instead, we we fill this gap by  utilizing the explicit form of the projection onto the probability simplex to establish the lower bound for the one-step improvement,
$$
\sum_{a\in \mathcal{A}}{\pi _{s,a}^{k+1}A_{s,a}^{k}}\ge\frac{\left( \underset{a\in \mathcal{A}}{\max}\,A_{s,a}^{k} \right) ^2}{ \underset{a\in \mathcal{A}}{\max}\,A_{s,a}^{k}+C}, \quad\forall s\in\mathcal{S}.
$$ 
Combining this result with the performance difference lemma yields that
$$
V^{k+1}\left( \rho \right) -V^k\left( \rho \right) \ge \mathcal{O} \left( \left( V^*\left( \rho \right) -V^k\left( \rho \right) \right) ^2 \right),$$
which directly implies the sublinear convergence of PPG.

Following the notation in the prototype  update, the key ingredient in our analysis is 
\begin{align*}
    f_s(\eta_s) := \sum_{a\in\calA} \pi^+_{s,a}(\eta_s) A^\pi_{s,a},
\end{align*}
where $\eta_s = \frac{\eta d^\pi_\mu(s)}{1-\gamma}$ for PPG. We first give an expression for $f_s(\eta_s)$.

\begin{lemma}[Improvement expression] \label{Improvement expression}
Consider the prototype update  in  \eqref{policy update of PPG and PQD}.  For any $\eta_s > 0$ one has

\begin{align*}
    f_s(\eta_s) &= \eta_s \left( \sum_{a\in\calB_s(\eta_s)} \left( A^\pi_{s,a} \right)^2 - \frac{1}{|\calB_s(\eta_s)|}\left( \sum_{a\in\calB_s(\eta_s)} A^\pi_{s,a} \right)^2 \right) \\
    &+ \sum_{a^\prime \in \calA \setminus \calB_s(\eta_s)} \pi_{s,a^\prime} \left( \frac{1}{|\calB_s(\eta_s)|} \sum_{a\in\calB_s(\eta_s)} (A^\pi_{s,a} - A^\pi_{s,a^\prime}) \right).
\end{align*}
\end{lemma}

%By lemma \ref{Improvement expression} we can give a lower bound of the improvement function $f_s(\eta_s)$.
\kw{The proof of Lemma~\ref{Improvement expression} is deferred to Section~\ref{sec: Improvement expression}. Based on this lemma, we are able to derive a lower bound for $f_s(\eta_s)$, as stated in the next lemma whose proof is deferred to Section~\ref{sec: improvement lower bound}.}
\begin{theorem}[Improvement lower bound] Consider the  update in \eqref{policy update of PPG and PQD}. For any $\eta_s>0$ one has

%%%%%%%%%%%%%%%%%%%%%%%%%%%%%%%%%
\iffalse
\begin{enumerate}
    \item (Monotonic increasing) $f_s(\eta_1) \ge f_s(\eta_2)$ for $\;\forall \eta_1 > \eta_2 > 0$.
    \item (Improvement lower bound) $\forall \eta > 0$, $\;
f_s\left( \eta  \right) \ge \frac{\min \left\{ \eta,1-\gamma \right\}}{2\left| \mathcal{A} \right|}\cdot \left( \underset{a\in \mathcal{A}}{\max}A_{s,a}^{\pi} \right) ^2$.
\end{enumerate}
\fi
%%%%%%%%%%%%%%%%%%%%%%%%%%%%%%%%%

\begin{align*}
    f_s(\eta_s) \geq \frac{\left(\max_a A^\pi_{s,a} \right)^2}{\max_a A^\pi_{s,a} + \frac{2+5|\mathcal{A}|}{\eta_s}}.
\end{align*}

\label{theorem:improvement lower bound for PPG}
\end{theorem}

With this lower bound, the sublinear convergence of PPG can be established together with the performance difference lemma.

\begin{theorem}[Sublinear convergence of PPG] With any constant step size $\eta_k = \eta$ and distribution $\rho \in \Delta(\calS)$, the policy sequence $\pi^k$ generated by PPG satisfies
\begin{align}\label{eq:sublinear}
    V^*(\rho) - V^k(\rho) \leq \frac{1}{k} \frac{1}{(1-\gamma)^2} \left\|\frac{d^*_\rho}{\rho}\right\|_{\infty} \left( 1 + \frac{2+5|\calA|}{\eta\tilde{\mu}} \right).
\end{align}
\label{theorem:sublinear of PPG}
\end{theorem}

\begin{remark} Compared with the previous results  in \textup{\cite{Xiao_2022,Agarwal_Kakade_Lee_Mahajan_2019,Zhang_Koppel_Bedi_Szepesvari_Wang_2020}}, the result in \eqref{eq:sublinear} removes the constraint $\eta_k\le \frac{1}{L}$  on the step size,  where $L = \frac{2\gamma \left| \mathcal{A} \right|}{\left( 1-\gamma \right) ^3}$ is the smoothness coefficient of the value function. The best sublinear convergence rate for PPG in prior works is achieved when $\eta_k = \frac{1}{L}$, leading to the result
\begin{align}
V^*\left( \mu \right) -V^{k}\left( \mu \right) \le \mathcal{O} \left( \frac{\left| \mathcal{S} \right|\left| \mathcal{A} \right|}{\left( 1-\gamma \right) ^5k}\left\| \frac{d_{\mu}^{*}}{\mu} \right\| _{\infty}^{2} \right).
\label{PPG_result_2}
\end{align}
By setting $\eta_k = \frac{1}{L} = \frac{\left( 1-\gamma \right) ^3}{2\gamma \left| \mathcal{A} \right|}$ and $\rho = \mu$ in \eqref{eq:sublinear} we can obtain the bound 
\begin{align*}
V^*\left( \mu \right) -V^k\left( \mu \right) &\le \frac{1}{k}\frac{1}{\left( 1-\gamma \right) ^2}\left\| \frac{d_{\mu}^{*}}{\mu} \right\| _{\infty}\left( 1+\frac{2\gamma \left| \mathcal{A} \right|\left( 2+5\left| \mathcal{A} \right| \right)}{\left( 1-\gamma \right) ^3\tilde{\mu}} \right)\\
&=\mathcal{O} \left( \frac{1}{k}\frac{\left| \mathcal{A} \right|^2}{\left( 1-\gamma \right) ^5}\left\| \frac{d_{\mu}^{*}}{\mu} \right\| _{\infty}\frac{1}{\tilde{\mu}} \right).
\end{align*}
Compared with \eqref{PPG_result_2}, the new bound has the same dependency on the discounted factor. In addition, the new bound \kw{is proportional to} $\left| \mathcal{A} \right|^2$ instead of $|\calS| |\calA|$, which is better in the case when $|\calS| > |\calA|$. 
Moreover,  Theorem~\ref{theorem:sublinear of PPG} suggests  that the best sublinear convergence rate for PPG is indeed achieved when  $\eta_k=\eta \ge \frac{2 + 5|\calA|}{\tilde{\mu}}$ rather than $\eta_k=\frac{1}{L}$, yielding the rate
$$
\mathcal{O} \left( \frac{1}{k}\frac{1}{(1-\gamma )^2}\left\| \frac{d_{\rho}^{*}}{\rho} \right\| _{\infty} \right).
$$
% which is better than prior results.

\end{remark}

\begin{remark}
 Our analysis technique is also available for the establishment of the sublinear convergence of PQA. However, the result is not as tight as the one obtained in \textup{\cite{Xiao_2022, Lan_2021}}  based on the particular structure of PQA within the framework of policy mirror ascent. Thus, we omit the details.  It is worth emphasizing again that, to the best of our knowledge, the analysis technique in \textup{\cite{Xiao_2022, Lan_2021}} for PQA is not applicable for PPG due the existence of the visitation measure. 
\end{remark}

\begin{proof}[Proof of Theorem~\ref{theorem:sublinear of PPG}]
By the performance difference lemma (Lemma \ref{performanceDifferenceLemma}) and Theorem \ref{theorem:improvement lower bound for PPG}, one has
\begin{align*}
&V^{k+1}\left( \rho \right) -V^k\left( \rho \right) =\frac{1}{1-\gamma}\sum_{s\in \mathcal{S}}{d_{\rho}^{k+1}\left( s \right) \sum_{a\in \mathcal{A}}{\pi _{s,a}^{k+1}A_{s,a}^{k}}} \ge
\mathbb{E} _{s\sim \rho}\left[ \frac{\left( \underset{a\in \mathcal{A}}{\max}\,\,A_{s,a}^{k} \right) ^2}{\underset{a\in \mathcal{A}}{\max}\,\,A_{s,a}^{k}+\frac{2+5\left| \mathcal{A} \right|}{\eta^k _{s}}} \right] 
\\               
\,\,                     &\overset{\left( a \right)}
\ge \mathbb{E} _{s\sim \rho}\left[ \frac{\left( \underset{a\in \mathcal{A}}{\max}\,\,A_{s,a}^{k} \right) ^2}{\underset{a\in \mathcal{A}}{\max}\,\,A_{s,a}^{k}+\frac{2+5\left| \mathcal{A} \right|}{\eta\tilde{\mu}}} \right] 
= 
\mathbb{E} _{s\sim \rho}\left[g\left( \underset{a\in \mathcal{A}}{\max}\,\,A_{s,a}^{k} \right) \right]
\\               
                        &\ge 
\left\| \frac{d_{\rho}^{*}}{\rho} \right\| _{\infty}^{-1}\mathbb{E} _{s\sim d_{\rho}^{*}}\left[ g\left( \underset{a\in \mathcal{A}}{\max}\,\,A_{s,a}^{k} \right) \right] 
\\
&\overset{\left( b \right)}
\ge \left\| \frac{d_{\rho}^{*}}{\rho} \right\| _{\infty}^{-1}\cdot g\left( \mathbb{E} _{s\sim d_{\rho}^{*}}\left[ \underset{a\in \mathcal{A}}{\max}\,\,A_{s,a}^{k} \right] \right),
\numberthis
\label{lower_bound_of_error_difference}
\end{align*}
where $g\left( x \right) :=\frac{x^2}{x+\frac{2+5|\calA|}{\eta\tilde{\mu}}}$ is a convex and monotonically increasing function when $x \ge 0$, (a) is due to $\eta_s^k = \frac{\eta d^k_\mu(s)}{1-\gamma} \geq \eta\tilde{\mu}$ according to inequality \eqref{bound of d}, and (b) is due to the Jensen Inequality. Notice that 
\begin{align*}
\mathbb{E} _{s\sim d_{\rho}^{*}}\left[ \underset{a\in \mathcal{A}}{\max}\,\,A_{s,a}^{k} \right] &\geq \mathbb{E} _{s\sim d_{\rho}^{*}}\left[ \mathbb{E} _{a\sim \pi _{s}^{*}}\left[ A_{s,a}^{k} \right] \right] =\left( 1-\gamma \right) \left[ \frac{1}{1-\gamma} \mathbb{E} _{s\sim d_{\rho}^{*}}\left[ \mathbb{E} _{a\sim \pi _{s}^{*}}\left[ A_{s,a}^{k} \right] \right] \right] \\
&=(1-\gamma) \left( V^*(\rho) - V^k(\rho)  \right). \numberthis \label{expectation_of_max_A}
\end{align*}    
Let $\delta_k:= V^*(\rho) - V^k(\rho)$. As $g$ is monotonically increasing, plugging \eqref{lower_bound_of_error_difference} into \eqref{expectation_of_max_A} yields that
\begin{align*}
\delta_k - \delta_{k+1}&\ge \left\| \frac{d_{\rho}^{*}}{\rho} \right\| _{\infty}^{-1}\cdot g\left( \left( 1-\gamma \right) \delta_k \right) 
=\left\| \frac{d_{\rho}^{*}}{\rho} \right\| _{\infty}^{-1}\cdot \frac{\left( 1-\gamma \right) ^2\delta _{k}^{2}}{\left( 1-\gamma \right) \delta _k+\frac{2+5\left| \mathcal{A} \right|}{\eta\tilde{\mu}}}. \numberthis \label{delta bound}
\end{align*}
Since $\delta_k \leq \frac{1}{1-\gamma}$ by Lemma~\ref{bounds of V Q A}, we have
\begin{align*}
    \delta_k - \delta_{k+1}&\ge \left\| \frac{d_{\rho}^{*}}{\rho} \right\| _{\infty}^{-1}\cdot \frac{\left( 1-\gamma \right) ^2\delta _{k}^{2}}{1+\frac{2+5\left| \mathcal{A} \right|}{\eta \tilde{\mu}}}.
\end{align*}
This inequality  implies that $\delta_k$ is monotonically decreasing. \kw{Dividing  both sides by $\delta_k^2$} yields
\begin{align*}
    \frac{1}{\delta_{k+1}} - \frac{1}{\delta_k} = \frac{\delta_{k}-\delta_{k+1}}{\delta_k \delta_{k+1}} \geq \frac{\delta_k - \delta_{k+1}}{\delta_k^2} \geq \left\| \frac{d_{\rho}^{*}}{\rho} \right\| _{\infty}^{-1}\cdot \frac{\left( 1-\gamma \right) ^2}{1+\frac{2+5\left| \mathcal{A} \right|}{\eta \tilde{\mu}}}.
\end{align*}
Consequently,
\begin{align*}
    \delta_k = \frac{1}{\frac{1}{\delta_k}} = \frac{1}{\delta_0 + \sum_{i=0}^{k-1}\left( \frac{1}{\delta_{k+1}} - \frac{1}{\delta_k} \right)}
    &\leq \frac{1}{\sum_{i=0}^{k-1}\left( \frac{1}{\delta_{k+1}} - \frac{1}{\delta_k} \right)} \\
    &\leq \frac{1}{k} \left\| \frac{d_{\rho}^{*}}{\rho} \right\| _{\infty} \left( \frac{1}{(1-\gamma)^2} + \frac{2+5|\calA|}{\eta\tilde{\mu}(1-\gamma)^2} \right),
\end{align*}
and the proof is complete.

\end{proof}

\subsection{Proof of Lemma \ref{Improvement expression}}\label{sec: Improvement expression}

Noting that
\begin{align*}
     1 = \sum_{a\in\calA}\pi_{s,a} &= \sum_{a\in\calB_s(\eta_s)} \pi^+_{s,a}(\eta_s) = \sum_{a\in\calB_s(\eta_s)} \left[ \pi_{s,a} + \eta_s A^\pi_{s,a} + \lambda_s(\eta_s) \right]_+ \\
     &= \sum_{a\in\calB_s(\eta_s)} \left[ \pi_{s,a} + \eta_s A^\pi_{s,a} + \lambda_s(\eta_s) \right],
\end{align*}
we have 
\begin{align*}
    \lambda_s(\eta_s) &= \frac{1}{|\calB_s(\eta_s)|} \left( 1 - \sum_{a\in\calB_s(\eta_s)} \left[ \pi_{s,a} + \eta_s A^\pi_{s,a} \right] \right)  \\
    &= \frac{1}{|\calB_s(\eta_s)|} \left( \sum_{a \in \calA \setminus \calB_s(\eta_s)} \pi_{s,a} - \eta_s \sum_{a\in\calB_s(\eta_s)}A^\pi_{s,a} \right). \numberthis \label{lambda expression}
\end{align*}
It follows that
\begin{align*}
    f_s(\eta_s) &= \sum_{a\in\calB_s(\eta_s)} \pi^+_{s,a}(\eta_s) A^\pi_{s,a} = \sum_{a\in\calB_s(\eta_s)} \left[ \pi_{s,a} + \eta_s A^\pi_{s,a} + \lambda_s(\eta_s) \right] A^\pi_{s,a} \\
    &\stackrel{(a)}{=} \lambda_s(\eta_s)\sum_{a\in\calB_s(\eta_s)} A^\pi_{s,a} + \eta_s \sum_{a\in\calB_s(\eta_s)} \left(A^\pi_{s,a}\right)^2 - \sum_{a^\prime\in\calA \setminus \calB_s(\eta_s)} \pi_{s,a}A^\pi_{s,a} \\
    &= \eta_s \left( \sum_{a\in\calB_s(\eta_s)} (A^\pi_{s,a})^2 - \frac{1}{|\calB_s(\eta_s)|} \left( \sum_{a\in\calB_s(\eta_s)} A^\pi_{s,a} \right)^2 \right) \\
    &+ \sum_{a^\prime\in\calA\setminus\calB_s(\eta_s)} \pi_{s,a^\prime} \left( \frac{1}{|\calB_s(\eta_s)|} \sum_{a\in\calB_s(\eta_s)} (A^\pi_{s,a} - A^\pi_{s,a^\prime}) \right),
\end{align*}
where $(a)$ utilizes the fact $\sum_{a\in\calB_s(\eta_s)} \pi_{s,a}A^\pi_{s,a} = 0$.

\subsection{Proof of Theorem \ref{theorem:improvement lower bound for PPG}}\label{sec: improvement lower bound}

Without loss of generality, we only consider the case $\calB_s(\eta_s)\setminus\calA_s^\pi\neq \emptyset$. First recall the expression of $f_s(\eta_s)$ in Lemma~\ref{Improvement expression}:
\begin{align*}
    f_s(\eta_s) &= \eta_s \left( \underset{I_1}{\underbrace{\sum_{a\in\calB_s(\eta_s)} (A^\pi_{s,a})^2 - \frac{1}{|\calB_s(\eta_s)|} \left( \sum_{a\in\calB_s(\eta_s)} A^\pi_{s,a} \right)^2}} \right) \\
    &+ \underset{I_2}{\underbrace{\sum_{a^\prime\in\calA\setminus\calB_s(\eta_s)} \pi_{s,a^\prime} \left( \frac{1}{|\calB_s(\eta_s)|} \sum_{a\in\calB_s(\eta_s)} (A^\pi_{s,a} - A^\pi_{s,a^\prime}) \right)}}.
\end{align*}
For the term $I_1$, it is \kw{evident that}
\begin{align*}
    I_1 &= |\calB_s(\eta_s)| \left( \mathbb{E}_{a\sim U} \left[ \left( A^\pi_{s,a} \right)^2 \right] - \left( \mathbb{E}_{a\sim U}\left[ A^\pi_{s,a} \right] \right)^2 \right) = |\calB_s(\eta_s)| \cdot \Var_{a\sim U} \left[ A^\pi_{s,a} \right],
\end{align*}
where $U$ denotes the uniform distribution on $\calB_s(\eta_s)$. Letting $\Delta_{s,a}^\pi := \underset{a^\prime\in\calA}{\max} \,A^\pi_{s,a^\prime} - A^\pi_{s,a}$,
\begin{align*}
    I_1 &= |\calB_{s}(\eta_s)| \cdot \mathrm{Var}_{a\sim U} \left[ A^\pi_{s,a} \right]= |\calB_{s}(\eta_s)| \cdot \mathrm{Var}_{a\sim U} \left[ \Delta^\pi_{s,a} \right] \\
    &= \sum_{a\in\calB_s(\eta_s)} (\Delta_{s,a}^\pi)^2 - \frac{1}{|\calB_s(\eta_s)|} \left( \sum_{a\in\calB_s(\eta_s)} \Delta_{s,a}^\pi \right) ^2 \\
    &\stackrel{(a)}{=}\sum_{a\in\calB_s(\eta_s)}(\Delta_{s,a}^\pi)^2 - \frac{1}{|\calB_s(\eta_s)|} \left( \sum_{\tilde{a}\in\calB_s(\eta_s) \setminus \calA_s^\pi} \Delta_{s,\tilde{a}}^\pi \right)^2 \\
    \\
    &\geq \sum_{a\in\calB_s(\eta_s)}(\Delta_{s,a}^\pi)^2 - \frac{|\calB_s(\eta_s) \setminus \calA^\pi_s|}{|\calB_{s}(\eta_s)|} \sum_{\tilde{a}\in\calB_s(\eta_s) \setminus \calA_s^\pi} (\Delta_{s,\tilde{a}}^\pi)^2 \
    \\
    \\
    &= \left( 1 - \frac{|\calB_s(\eta_s) \setminus \calA^\pi_s|}{|\calB_{s}(\eta_s)|} \right) \sum_{a\in\calB_s(\eta_s)}(\Delta_{s,a}^\pi)^2= \frac{|\calB_s(\eta_s) \cap A^\pi_s|}{|\calB_s(\eta_s)|} \sum_{a\in\calB_s(\eta_s)}\left(\Delta_{s,a}^\pi\right)^2 \\
    &\geq \frac{1}{|\calB_s(\eta_s)|} \sum_{a\in\calB_s(\eta_s)}\left(\Delta_{s,a}^\pi\right)^2,
\end{align*}
where $(a)$ leverages the property that $\Delta_{s,a}^\pi = 0$ for $\pi$-optimal actions $a \in \calA_s^\pi$. 

For the term $I_2$, we can rewrite it through the notation $\Delta^\pi$ as follows:
\begin{align*}
    I_2 = \sum_{a^\prime \in \calA \setminus \calB_s(\eta_s)} \pi_{s,a^\prime} \left(\Delta_{s,a^\prime}^\pi - \bar{\Delta}_s^\pi\right),
\end{align*}
where $\bar{\Delta}_s^\pi := \frac{1}{|\calB_s(\eta_s)|} \sum_{a\in\calB_s(\eta_s)} \Delta^\pi_{s,a}$.
By lemma \ref{proj_property_2}, for any action $a^\prime \notin \calB_s(\eta)$ we have
\begin{align*}
    &\sum_{a\in\calB_s(\eta)}\hspace{-0.2cm} \left[\pi_{s,a} + \eta A^\pi_{s,a} - (\pi_{s,a^\prime} + \eta A^\pi_{s,a^\prime})\right]
    \geq \hspace{-.3cm}\sum_{a\in\calB_s(\eta)} \hspace{-.2cm}\left[\pi_{s,a} + \eta A^\pi_{s,a} - \max_{a\prime\not\in\calB_s(\eta)}(\pi_{s,a^\prime} + \eta A^\pi_{s,a^\prime})\right] \\
    &\stackrel{(a)}{=} \sum_{a\in\calB_s(\eta)} \left[\pi_{s,a} + \eta A^\pi_{s,a} - \max_{a\prime\not\in\calB_s(\eta)}(\pi_{s,a^\prime} + \eta A^\pi_{s,a^\prime})\right]_+ \geq 1,
\end{align*}
where $(a)$ is due to $a\in \calB_s(\eta)$ and $a^\prime \not\in \calB_s(\eta)$, implying $\pi_{s,a} + \eta A^\pi_{s,a} > \pi_{s,a^\prime} + \eta A^\pi_{s,a^\prime}$. It follows that
\begin{align*}
1&\le \sum_{a\in\calB_s(\eta)} \left[ \pi_{s,a} - \pi_{s,a^\prime} + \eta (A^\pi_{s,a} - A^\pi_{s,a^\prime}) \right] \\
&= \eta \sum_{a\in\calB_s(\eta)} (A^\pi_{s,a} - A^\pi_{s,a^\prime}) - |\calB_s(\eta)| \pi_{s,a^\prime} + \sum_{a\in\calB_s(\eta)} \pi_{s,a},
\end{align*}
which yields 
\begin{align*}
\frac{1}{\left| \calB _s\left( \eta \right) \right|}\sum_{a\in \calB _s\left( \eta \right)} \left( {A_{s,a}^{\pi}}-A_{s,a^\prime}^{\pi} \right) \ge \frac{1}{\eta\left| \calB _s\left( \eta \right) \right|}\left( 1-\sum_{a\in \calB _s\left( \eta \right)}{\pi _{s,a}}+\left| \calB _s\left( \eta \right) \right|\pi _{s,a^\prime} \right) \geq\frac{\pi _{s,a^\prime}}{\eta}.
\end{align*} 
Using the notation of $\Delta^\pi_{s,a}$ and $\bar{\Delta}^\pi_s$, this inequality can be reformulated as
\begin{align}
    \forall a^\prime \in \calA\setminus\calB_s(\eta_s): \quad \Delta_{s,a^\prime}^\pi - \bar{\Delta}_s^\pi \ge \frac{\pi_{s,a^\prime}}{\eta_s}.
    \label{ieq: 1}
\end{align}
Let $\calB_s(0) := \mathrm{supp}(\pi_s) = \left\{ a: \pi_{s,a} > 0 \right\}$. By \eqref{ieq: 1} we know that
\begin{align}
    \forall a^\prime \in \left(\calA \setminus \calB_s(\eta_s)\right) \cap \calB_s(0): \quad \Delta_{s,a^\prime}^\pi - \bar{\Delta}_s^\pi \ge \frac{\pi_{s,a^\prime}}{\eta_s} > 0.
    \label{ieq: 1 argument}
\end{align}
Furthermore, 
\begin{align*}
    I_2 = \sum_{a^\prime \in \calA \setminus \calB_s(\eta_s)} \pi_{s,a^\prime} \left(\Delta_{s,a^\prime}^\pi - \bar{\Delta}_s^\pi\right) = \sum_{a^\prime \in \left(\calA \setminus \calB_s(\eta_s)\right) \cap \calB_s(0)} \pi_{s,a^\prime}\left(\Delta_{s,a^\prime}^\pi - \bar{\Delta}_s^\pi\right).
\end{align*}

Combining $I_1$ and $I_2$ together, we have
\begin{align*}
    f_s(\eta_s) \geq \frac{\eta_s}{|\calB_s(\eta_s)|} \sum_{a\in\calB_s(\eta_s)} \left( \Delta_{s,a}^\pi \right)^2 + \sum_{a^\prime \in \left(\calA \setminus \calB_s(\eta_s)\right) \cap \calB_s(0)} \pi_{s,a^\prime} (\Delta_{s,a^\prime}^\pi - \bar{\Delta}_s^\pi).
\end{align*}
By Cauchy-Schwarz Inequality, 
\begin{align*}
&f_s \left( \eta_s \right) \times  \left( \frac{\left| \mathcal{B} _s\left( \eta_s \right) \right|}{\eta_s}\sum_{a\in \mathcal{B} _s\left( \eta_s \right)}{\left( \pi _{s,a} \right) ^2}+\sum_{a^\prime\in \left(\calA \setminus \calB_s(\eta_s)\right) \cap \calB_s(0)}{\pi _{s,a^\prime}}\frac{\left( \Delta _{s,a^\prime}^{\pi} \right) ^2}{\Delta _{s,a^\prime}^{\pi}-\bar{\Delta}_{s}^{\pi}} \right) \\
&\ge \left( \frac{\eta_s}{\left| \mathcal{B} _s\left( \eta_s \right) \right|}\sum_{a\in \mathcal{B} _s\left( \eta_s \right)}{\left( \Delta _{s,a}^{\pi} \right) ^2}+\sum_{a^\prime\in \left(\calA \setminus \calB_s(\eta_s)\right) \cap \calB_s(0)}{\pi _{s,a^\prime}}\left( \Delta _{s,a^\prime}^{\pi}-\bar{\Delta}_{s}^{\pi} \right) \right)     \\    &\times \left( \frac{\left| \mathcal{B} _s\left( \eta_s \right) \right|}{\eta_s}\sum_{a\in \mathcal{B} _s\left( \eta_s \right)}{\left( \pi _{s,a} \right) ^2}+\sum_{a^\prime\in \left(\calA \setminus \calB_s(\eta_s)\right) \cap \calB_s(0)}{\pi _{s,a^\prime}}\frac{\left( \Delta _{s,a^\prime}^{\pi} \right) ^2}{\Delta _{s,a^\prime}^{\pi}-\bar{\Delta}_{s}^{\pi}} \right) 
\\
\,\,                                                                        &\ge \left( \sum_{a\in \mathcal{B} _s\left( \eta_s \right)}{\pi _{s,a}\Delta _{s,a}^{\pi}}+\sum_{a^\prime\in \left(\calA \setminus \calB_s(\eta_s)\right) \cap \calB_s(0)}{\pi _{s,a^\prime}}\Delta _{s,a^\prime}^{\pi} \right) ^2
{=}\left( \underset{a\in \mathcal{A}}{\max}\,\,A_{s,a}^{\pi} \right) ^2.
\end{align*}
Therefore, we can obtain $$f_s(\eta_s) \geq \frac{1}{G}\left( \underset{a\in\calA}{\max}\calA^\pi_{s,a} \right)^2,$$ where
\begin{align*}
G:=\underset{G_1}{\underbrace{\frac{\left| \mathcal{B} _s\left( \eta_s \right) \right|}{\eta_s}\sum_{a\in \mathcal{B} _s\left( \eta_s \right)}{\left( \pi _{s,a} \right) ^2}}}+\underset{G_2}{\underbrace{\sum_{a\in \left(\calA \setminus \calB_s(\eta_s)\right) \cap \calB_s(0)}{\pi _{s,a}}\frac{\left( \Delta _{s,a}^{\pi} \right) ^2}{\Delta _{s,a}^{\pi}-\bar{\Delta}_{s}^{\pi}}}}.
\end{align*}
Next we will give an upper bound of $G$. For the term $G_1$, \kw{it is straightforward to see that}
$$
G_1<\frac{\left| \mathcal{B} _s\left( \eta_s \right) \right|}{\eta_s}\left( \sum_{a\in \mathcal{B} _s\left( \eta_s \right)}{\pi _{s,a}} \right) ^2<\frac{\left| \mathcal{A} \right|}{\eta_s}.$$
For the term $G_2$, a direct computation yields that,
\begin{align*}
G_2&=\sum_{a\in \left(\calA \setminus \calB_s(\eta_s)\right) \cap \calB_s(0)}{\pi _{s,a}}\frac{\left( \Delta _{s,a}^{\pi} \right) ^2}{\Delta _{s,a}^{\pi}-\bar{\Delta}_{s}^{\pi}}
\\
\,\,   &=\sum_{a\in \left(\calA \setminus \calB_s(\eta_s)\right) \cap \calB_s(0)}{\pi _{s,a}}\frac{\left( \Delta _{s,a}^{\pi} \right) ^2-\left( \bar{\Delta}_{s}^{\pi} \right) ^2+\left( \bar{\Delta}_{s}^{\pi} \right) ^2}{\Delta _{s,a}^{\pi}-\bar{\Delta}_{s}^{\pi}}
\\
\,\,   &=\sum_{a\in \left(\calA \setminus \calB_s(\eta_s)\right) \cap \calB_s(0)}{\pi _{s,a}\left( \Delta _{s,a}^{\pi}+\bar{\Delta}_{s}^{\pi}+\frac{\left( \bar{\Delta}_{s}^{\pi} \right) ^2}{\Delta _{s,a}^{\pi}-\bar{\Delta}_{s}^{\pi}} \right)}
\\
\,\,   &\le\sum_{a\in \mathcal{A}}{\pi _{s,a}\Delta _{s,a}^{\pi}}+\bar{\Delta}_{s}^{\pi}\sum_{a\in \left(\calA \setminus \calB_s(\eta_s)\right) \cap \calB_s(0)}{\pi _{s,a}}+\left( \bar{\Delta}_{s}^{\pi} \right) ^2\sum_{a\in \left(\calA \setminus \calB_s(\eta_s)\right) \cap \calB_s(0)}{\frac{\pi _{s,a}}{\Delta _{s,a}^{\pi}-\bar{\Delta}_{s}^{\pi}}}
\\
\,\,   &=\underset{a\in \mathcal{A}}{\max}\,\,A_{s,a}^{\pi}+\bar{\Delta}_{s}^{\pi}\sum_{a\in \left(\calA \setminus \calB_s(\eta_s)\right) \cap \calB_s(0)}{\pi _{s,a}}+\left( \bar{\Delta}_{s}^{\pi} \right) ^2\sum_{a\in \left(\calA \setminus \calB_s(\eta_s)\right) \cap \calB_s(0)}{\frac{\pi _{s,a}}{\Delta _{s,a}^{\pi}-\bar{\Delta}_{s}^{\pi}}}.
\numberthis
\label{upper_bound_of_G2}
\end{align*}
Lemma \ref{lemma: advantage lower bound for support actions} shows that
\begin{align}
    \forall a\in\calB_s(\eta_s): \Delta_{s,a}^\pi \leq \frac{2\pi_s(\calA \setminus \calA_s^\pi)}{\eta_s} \leq \frac{2}{\eta_s} \quad
    \Longrightarrow \quad \bar{\Delta}_s^\pi \leq \frac{2}{\eta_s}.
    \label{ieq: 2}
\end{align}
Plugging \eqref{ieq: 1 argument} and \eqref{ieq: 2} into \eqref{upper_bound_of_G2} we have
\begin{align*}
G_2&\le \underset{a\in \mathcal{A}}{\max}\,\,A_{s,a}^{\pi}+\frac{2}{\eta_s}\sum_{a\in \left(\calA \setminus \calB_s(\eta_s)\right) \cap \calB_s(0)}{\pi _{s,a}}+\left( \frac{2}{\eta_s} \right) ^2\sum_{a\in \left(\calA \setminus \calB_s(\eta_s)\right) \cap \calB_s(0)}{\frac{\pi _{s,a}}{\frac{1}{\eta_s}\pi _{s,a}}}
\\
\,\, &=\underset{a\in \mathcal{A}}{\max}\,\,A_{s,a}^{\pi}+\frac{2}{\eta_s}\sum_{a\in \left(\calA \setminus \calB_s(\eta_s)\right) \cap \calB_s(0)}{\pi _{s,a}}+\frac{4}{\eta_s}\sum_{a\in \left(\calA \setminus \calB_s(\eta_s)\right) \cap \calB_s(0)}{1}
\\
\,\, &\le \underset{a\in \mathcal{A}}{\max}\,\,A_{s,a}^{\pi}+\frac{2+4\left| \mathcal{A} \right|}{\eta_s}.
\end{align*}
Thus we can finally obtain
$$
G=G_1+G_2\le \underset{a\in \mathcal{A}}{\max}\,\,A_{s,a}^{\pi}+\frac{2+5\left| \mathcal{A} \right|}{\eta_s},
$$
and 
\begin{align*}
    f_s(\eta_s) \geq \frac{\left(\max_a A^\pi_{s,a} \right)^2}{G} \geq \frac{\left(\max_a A^\pi_{s,a} \right)^2}{\max_a A^\pi_{s,a} + \frac{2+5|\mathcal{A}|}{\eta_s}}. \numberthis
    \label{lower bound of f}
\end{align*}

%For the case that $\hat{\calB} _s(\eta_s) = \calA$ (which includes the case $\calB_s(\eta_s) = \calA$), a similar derivation gives that $f_s(\eta_s) \geq {\left(\max_a A^\pi_{s,a}\right)^2}/{G_1}$, which satisfies the lower bound \eqref{lower bound of f} above as well.

\section{Finite iteration convergence results}
\label{Finite time convergence results}

\subsection{Finite iteration convergence of PPG and PQA}
\label{Finite time Convergence of PPG and PQD with constant step size}

In this section, we show that both PPG and PQA output an optimal policy after a finite iteration $k_0$ and we will use the sublinear analysis (Theorem \ref{theorem:sublinear of PPG} for PPG and \eqref{xiao's results for PQA} for PQA) to derive an upper bound of $k_0$.
\kw{The overall idea is first sketched as follows.}
For an arbitrary $s \in \mathcal{S}$, letting $\mathcal{B} = \mathcal{A}^*_s$, $\mathcal{C} = \calA \setminus \mathcal{A}^*_s$ in Lemma \ref{proj_property_2} (recall that $\calA^*_s$ is the set of optimal actions, i.e. $\pi^*$-optimal actions), we have
\begin{align}
\forall a^{\prime}\notin \mathcal{A} _{s}^{*}, \;\; \pi _{s,a^{\prime}}^{+}=0 \;\Longleftrightarrow\; \sum_{a\in \mathcal{A} _{s}^{*}}{\left( \pi _{s,a}+\eta _sA_{s,a}^{\pi}-\underset{a^{\prime}\notin \mathcal{A} _{s}^{*}}{\max}\left( \pi _{s,a^{\prime}}+\eta _sA_{s,a^{\prime}}^{\pi} \right) \right) _+}\ge 1.
\label{finite_time_convergence_intuition_formualtion}
\end{align}
By the definition of $b^\pi_s$ and $\Delta$,
\begin{align*}
    b_{s}^{\pi}=\pi _s\left( \mathcal{A} \setminus \mathcal{A} _{s}^{*} \right), \quad \quad \Delta =\underset{s\in \tilde{S},a \notin \mathcal{A} _{s}^{*}}{\min}\left| A^{*}(s,a) \right|,
\end{align*}
when $V^\pi$ is sufficiently close to $V^*$, we know that for any $a^\prime \notin \calA^*_s$,
$$
\sum_{a\in \mathcal{A} _{s}^{*}}{\pi _{s,a}}\approx 1, \;\; A_{s,a}^{\pi}-A_{s,a^{\prime}}^{\pi}\approx A_{s,a}^{*}-A_{s,a^\prime}^{*}\ge \Delta. 
$$
Since $\pi_{s,a'}\le b_s^\pi$,  we asymptotically have
$$
\sum_{a\in \mathcal{A} _{s}^{*}}{\left( \pi _{s,a}+\eta _sA_{s,a}^{\pi}-\underset{a^{\prime}\notin \mathcal{A} _{s}^{*}}{\max}\left( \pi _{s,a^{\prime}}+\eta _sA_{s,a^{\prime}}^{\pi} \right) \right) _+} \ge 1-\mathcal{O} \left( b_{s}^{\pi} \right) +\mathcal{O} \left( \Delta \right).
$$
This implies that if  $b^\pi_s$ is sufficiently small, the condition in \eqref{finite_time_convergence_intuition_formualtion} will be met. Then both  PPG and PQA will output the optimal policy.

\begin{lemma}[Optimality condition] Consider the prototype update in \eqref{policy update of PPG and PQD}. Define
$$\varepsilon^\pi_{s,a} := \eta_s \left( A^\pi_{s,a} - A^*_{s,a} \right)\quad \mbox{and}\quad\varepsilon^\pi_s := [\varepsilon_{s,a}^\pi]_{a \in \mathcal{A}}.$$
 If the input policy $\pi$ satisfies,
    \begin{align}
      b_{s}^{\pi}+\|\varepsilon _{s}^{\pi}\|_{\infty}\le \frac{\eta _{s}\Delta}{2},\quad \forall s\in \mathcal{S},
    \label{finite time condition of PPG and PQD }
    \end{align}
    then $\pi^+$ is an optimal policy. %$\pi^+=\pi^*$, i.e., $\pi^+_s = \pi^*_s$  for all $s \in \calS$.
    \label{lemma: condition of finite time of PPG and PQD}
\end{lemma}

\begin{proof}
For any $s \in \calS$, a direction computation yields that
    \begin{align*}
&\sum_{a\in \mathcal{A} _{s}^{*}}{\left( \pi _{s,a}+\eta _{s}A_{s,a}^{\pi}-\max_{a^{\prime}\notin \mathcal{A} _{s}^{*}} \left( \pi _{s,a^{\prime}}+\eta _{s}A_{s,a^{\prime}}^{\pi} \right) \right) _+} \\
&\ge \sum_{a\in \mathcal{A} _{s}^{*}}{\left( \pi _{s,a}+\eta _{s}A_{s,a}^{\pi}-\max_{a^{\prime}\notin \mathcal{A} _{s}^{*}} \left( \pi _{s,a^{\prime}}+\eta _{s}A_{s,a^{\prime}}^{\pi} \right) \right)}
\\
&\overset{\left( a \right)}{=}\sum_{a\in \mathcal{A} _{s}^{*}}{\left( \pi _{s,a}+\eta _{s}A_{s,a}^{\pi}-\left( \pi _{s,\tilde{a}}+\eta_sA_{s,\tilde{a}}^{\pi} \right) \right)}
\\
&=\sum_{a\in \mathcal{A} _{s}^{*}}{\left[ \left( \pi_{s,a}-\eta_s\left| A_{s,a}^{*} \right|+\varepsilon _{s,a}^{k} \right) -\left( \pi _{s,\tilde{a}}-\eta_s\left| A_{s,\tilde{a}}^{*} \right|+\varepsilon _{s,\tilde{a}}^{\pi} \right) \right]}
\\
&\overset{\left( b \right)}{=}\sum_{a\in \mathcal{A} _{s}^{*}}{\left[ \left( \pi_{s,a}+\varepsilon _{s,a}^{k} \right) -\left( \pi _{s,\tilde{a}}-\eta_s\left| A_{s,\tilde{a}}^{*} \right|+\varepsilon _{s,\tilde{a}}^{\pi} \right) \right]}
\\
&\overset{\left( c \right)}{\ge}\sum_{a\in \mathcal{A} _{s}^{*}}{\left[ \left( \pi_{s,a}-\pi _{s,\tilde{a}} \right) +\eta_s\Delta +\left( \varepsilon _{s,a}^{\pi}-\varepsilon _{s,\tilde{a}}^{\pi} \right) \right]}
\ge \sum_{a\in \mathcal{A} _{s}^{*}}{\left[ \left( \pi_{s,a}-b_{s}^{\pi} \right) +\eta_s\Delta -2\left\| \varepsilon _{s}^{\pi} \right\| _{\infty} \right]}
\\
&=\sum_{a\in \mathcal{A} _{s}^{*}}{\pi_{s,a}}+\left| \mathcal{A} _{s}^{*} \right|\left( \eta_s\Delta -b_{s}^{\pi}-2\left\| \varepsilon _{s}^{\pi} \right\| _{\infty} \right) 
=1-b_{s}^{\pi}+\left| \mathcal{A} _{s}^{*} \right|\left( \eta_s\Delta -b_{s}^{\pi}-2\left\| \varepsilon _{s}^{\pi} \right\| _{\infty} \right) 
\\
&\ge 1-\left| \mathcal{A} _{s}^{*} \right|b_{s}^{\pi}+\left| \mathcal{A} _{s}^{*} \right|\left( \eta_s\Delta -b_{s}^{\pi}-2\left\| \varepsilon _{s}^{\pi} \right\| _{\infty} \right) 
=1+\left| \mathcal{A} _{s}^{*} \right|\left[ \eta_s\Delta -2\left( b_{s}^{\pi}+\left\| \varepsilon _{s}^{\pi} \right\| _{\infty} \right) \right] 
\ge 1,
    \end{align*}
    where   $
\tilde{a}:=\mathrm{argmax}_{a^{\prime}\notin \mathcal{A} _{s}^{*}} \,\,\left\{ \pi _{s,a^{\prime}}^{k}+\eta_sA_{s,a^{\prime}}^{k} \right\} $ in $(a)$, $(b)$ is due to $A^*_{s,a}=0$ for all $a \in \mathcal{A}^*_s$, and $(c)$ follows from the definition of $\Delta$. Combining this result with Lemma \ref{proj_property_2} we obtain that 
\ljc{
    \begin{align*}
 \pi _{s,a'}^{+}=0,\quad \forall\, a'\notin \mathcal{A} _{s}^{*}.
    \end{align*}
}
    which means $\pi^+$ is an optimal policy.
\end{proof}

Next we will show that the LHS of  \eqref{finite time condition of PPG and PQD } is actually of  order $\mathcal{O}\left( \left\| V^*-V^k \right\| _{\infty} \right).$ %\ljc{Since the RHS of \eqref{finite time condition of PPG and PQD } satisfies
%$$
%\frac{\eta _s\Delta}{2}=\frac{1}{2}\frac{\eta d_{\mu}^{\pi}\left( s \right)}{1-\gamma}\Delta \ge %\frac{\eta \tilde{\mu}}{2}\Delta >0,
%$$
%} for any constant step size,  
Thus the condition \eqref{finite time condition of PPG and PQD } can be satisfied provided the value error converges to zero and step sizes are constant (in this case the RHS of \eqref{finite time condition of PPG and PQD } is $\mathcal{O}(\Delta)$). 

\begin{lemma}[Optimality condition continued in terms of state values] Consider the prototype update in  \eqref{policy update of PPG and PQD}. If the state values of the input policy $\pi$ satisfies,
    \begin{align}
   \left\| V^*-V^\pi \right\| _{\infty}\le \frac{\Delta}{2}\frac{\eta _{s}\Delta}{1+\eta _{s}\Delta},
    \quad\forall s\in \mathcal{S},
    \label{finite time condition of PPG and PQD w.r.t value error}
    \end{align}
    then $\pi^+$ is an optimal policy. %$\pi^+=\pi^*$, i.e. $\pi^+_s = \pi^*_s$ for all $s \in \calS$.
    \label{corollary: condition of finite time of PPG and PQD w.r.t value error}
\end{lemma}
\begin{proof} For any $s \in \mathcal{S}$, setting $
\rho \left( \cdot \right) =\mathbb{I} \left( \cdot =s \right) $
in Lemma \ref{lemma: b le V error}, where $\mathbb{I}$ is the indicator function, yields that
\begin{align*}
b_{s}^{\pi}\le \frac{V^*\left( s \right) -V^\pi\left( s \right)}{\Delta}\le \frac{\left\| V^*-V^\pi \right\| _{\infty}}{\Delta}.
\end{align*}
Combining this result with Lemma \ref{lemma: advantage_Lemma} we have
\begin{align}
\underset{s\in \mathcal{S}}{\max}\,\,b_{s}^{\pi}+\left\| \varepsilon _{s}^{\pi} \right\| _{\infty}\le \frac{1}{\Delta}\left\| V^*-V^\pi \right\| _{\infty}+\eta _{s}\left\| V^*-V^\pi \right\| _{\infty}=\left( \frac{1}{\Delta}+\eta _{s} \right) \left\| V^*-V^\pi \right\| _{\infty}.
\label{formulation1}
\end{align}
The proof is completed by noting the assumption and  Lemma~\ref{lemma: condition of finite time of PPG and PQD}.
\end{proof}

Since the sublinear convergence of PPG (Theorem \ref{theorem:sublinear of PPG}) and PQA (\eqref{xiao's results for PQA}) has already been established, there must exist an iteration $k_0$ such that $\left\| V^*-V^k \right\|_{\infty}$ is smaller than the threshold given in Lemma~\ref{corollary: condition of finite time of PPG and PQD w.r.t value error}. 

\begin{theorem}[Finite iteration convergence of PPG]
With any constant step size $\eta_k = \eta > 0$, PPG terminates after at most 
\label{upper bound of k0 with constant stepsize for PPG}
\begin{align*}
    k_0:=\left\lceil\frac{2}{\Delta} \left( 1 + \frac{1}{\eta\tilde{\mu}\Delta} \right) \frac{1}{\tilde{\mu}(1-\gamma)^2} \left\| \frac{d^*_\mu}{\mu} \right\|_\infty \left( 1 + \frac{2+5|\calA|}{\eta\tilde{\mu}} \right)\right\rceil
\end{align*}
iterations.
\end{theorem}
\begin{proof} Since $\eta^k_s = \frac{\eta d^k_\mu(s)}{1-\gamma} > \eta \tilde{\mu}$ for PPG,  the RHS of \eqref{finite time condition of PPG and PQD w.r.t value error} satisfies 
\begin{align*}
\frac{\Delta}{2}\frac{\eta _{s}^{k}\Delta}{1+\eta _{s}^{k}\Delta}\ge \frac{\Delta}{2}\frac{\eta \tilde{\mu}\Delta}{1+\eta \tilde{\mu}\Delta}.\numberthis\label{eq:kwtmp01}
\end{align*}
According to Lemma \ref{lemma: advantage_Lemma} and Theorem \ref{theorem:sublinear of PPG},
\begin{align*}
\left\| V^*-V^k \right\| _{\infty}&\le \frac{V^*\left( \mu \right) -V^k\left( \mu \right)}{\tilde{\mu}}\le \frac{1}{k} \frac{1}{(1-\gamma)^2} \left\| \frac{d^*_\mu}{\mu} \right\|_\infty \frac{1}{\tilde{\mu}} \left( 1 + \frac{2+5|\calA|}{\eta\tilde{\mu}} \right)\\
&\leq \frac{\Delta}{2}\frac{\eta _{s}\Delta}{1+\eta _{s}\Delta},
\numberthis\label{upper bound of infty_error_of_PPG}
\end{align*}
where the last inequality follows from \eqref{eq:kwtmp01} and the expression of $k_0$.
\end{proof}

\begin{theorem}[Finite iteration convergence of PQA]
\label{upper bound of k0 with constant stepsize}
With any constant step size $\eta_k = \eta >0$, PQA terminates after at most  
$$
k_0:=\left\lceil\frac{2}{\Delta}\left( 1+\frac{1}{\eta \Delta} \right)\left( \frac{1}{\eta \left( 1-\gamma \right)}+\frac{1}{\left( 1-\gamma \right) ^2} \right) - 1\right\rceil
$$
iterations.
\end{theorem}
\begin{proof}
%Consider %$
%s_k=\underset{s\in \mathcal{S}}{\mathrm{arg}\max}\left\{ V^*\left( s \right) -V^{k}\left( s \right) \right\}$, 
Note that the following sublinear convergence of PQA has been established in \cite{Xiao_2022} for any constant step size,
\begin{align}
    V^*\left( \rho \right) -V^{k}\left( \rho \right) \le \frac{1}{k+1}\left( \frac{\mathbb{E} _{s\sim d_{\rho}^{*}}\left[ \left\| \pi _{s}^{*}-\pi _{s}^{0} \right\| _{2}^{2} \right]}{2\eta \left( 1-\gamma \right)}+\frac{1}{\left( 1-\gamma \right) ^2} \right).
    \label{xiao's results for PQA}
\end{align}
Plugging $
\rho _s\left( \cdot \right) :=\mathbb{I}\left\{ \cdot =s \right\}$ into \eqref{xiao's results for PQA} yields that 
\begin{align*}
%\left\| V^*-V^{k} \right\| _{\infty}&=
V^*\left( s \right) -V^{k}\left( s \right) 
 &\le \frac{1}{k+1}\left( \frac{\mathbb{E} _{s\sim d_{\rho _s}^{*}}\left[ \left\| \pi _{s}^{*}-\pi _{s}^{0} \right\| _{2}^{2} \right]}{2\eta \left( 1-\gamma \right)}+\frac{1}{\left( 1-\gamma \right) ^2} \right) 
\\
&\le \frac{1}{k+1}\left( \frac{1}{\eta \left( 1-\gamma \right)}+\frac{1}{\left( 1-\gamma \right) ^2} \right).
%\numberthis
%\label{upper bound of infty_error of PQD}
\end{align*}
Since this bound holds for any $s$, it also holds for $\|V^*-V^k\|_\infty$. Then it can be easily verified that 
 the condition in Lemma~\ref{corollary: condition of finite time of PPG and PQD w.r.t value error} is satisfied given the expression of $k_0$.
\end{proof}
%%%%%%%%%%
{
Before proceeding, we give two short discussions on the finite iteration convergence of PPG and PQA. Firstly, it will be shown that a condition similar to that in Lemma~\ref{lemma: condition of finite time of PPG and PQD}  can be obtained based on the optimality condition of the optimization problem. Secondly, though the finite iteration convergence for the  homotopic PQA is discussed in \cite{Li_Zhao_Lan_2022}, it does not imply the finite iteration convergence of PQA for any constant step size. A simple bandit example is used to illustrate that the homotopic PQA requires sufficiently large step size to converge (in fact, the finite iteration of the homotopic PQA is established  for exponentially increasing step sizes in \cite{Li_Zhao_Lan_2022}).
%%%%%%%%
\subsubsection{Short discussion I}
Recall that the update \eqref{policy update of PPG and PQD} corresponds to the following optimization: 
\begin{align*}
    \pi^+_s \hspace{-.1cm}=\hspace{-.05cm} \argmax_{p\in\Delta(\mathcal{A})}\left\{\hspace{-.05cm}\eta_s\langle Q^\pi(s,\cdot),p\rangle\hspace{-.05cm}-\hspace{-.05cm}\frac{1}{2}\|p-\pi_s\|_2^2\right\}\hspace{-.1cm}=\hspace{-.05cm}\argmax_{p\in\Delta(\mathcal{A})}\left\{\hspace{-.05cm}\eta_s\langle A^\pi(s,\cdot),p\rangle\hspace{-.05cm}-\hspace{-.05cm}\frac{1}{2}\|p-\pi_s\|_2^2\right\}.
\end{align*}
The optimality condition for this problem is given by (see for example \cite{Rockafellar})
\begin{align*}
\langle\eta_s A^\pi(s,\cdot)-\pi_s^++\pi_s, \; p'-\pi_s^+\rangle\leq 0,\quad \forall\, p'\in\Delta(\mathcal{A}).\numberthis\label{eq:opt-cond-001}
\end{align*}
Define $N_{\Delta}(p)$ as the normal cone of $\Delta(\mathcal{A})$ at $p$,
\begin{align*}
    N_\Delta(p)=\{g~|~ g^T(p'-p)\leq 0,\,\forall\,p'\in\Delta(\mathcal{A})\}.
\end{align*}
The condition in \eqref{eq:opt-cond-001} can be equivalently expressed as
\begin{align*}
\eta_s A^\pi(s,\cdot)-\pi_s^++\pi_s\in N_\Delta(\pi_s^+).
\end{align*}
Moreover, note that (see for example \cite{Beck})
\begin{align*}
N_\Delta(\pi_s^+) = \left\{(g_1,\cdots,g_{|\mathcal{A}|})~|~g_i\leq g_j=g_\ell,\,\forall\, i\not\in\mathrm{supp}(\pi_s^+)\mbox{ and }\forall\,j,\ell\in\mathrm{supp}(\pi_s^+)\right\}.
\end{align*}
%Therefore, if $\forall\, s\in\mathcal{S},\, a\in\mathcal{A}_s^*$ and $a'\not\in\mathcal{A}_s^*$, it can be shown that
Therefore, if $\forall\, s\in\mathcal{S}$, it can be shown that there exists $g^\pi_{s, \cdot}\in N_\Delta(\pi_s^+)$, such that $\forall \, a\in\mathcal{A}_s^*$ and $a'\not\in\mathcal{A}_s^*$, 
\begin{align*}
   g^\pi_{s,a}-g^\pi_{s,a'}= \left(\eta_s A^\pi_{s,a}-\pi_{s,a}^++\pi_{s,a}\right)-\left(\eta_s A^\pi_{s,a'}-\pi_{s,a'}^++\pi_{s,a'}\right)>0,\numberthis\label{eq:final-cond}
\end{align*}
we can conclude that 
\begin{align*}
    \forall s\in\mathcal{S},\,a'\not\in\mathcal{A}_s^*: \quad  a'\not\in\mathrm{supp}(\pi_s^+),
\end{align*}
which implies $\pi^+$ is an optimal policy.

Recalling the definition of $\varepsilon^\pi_{s,a}=\eta_s \left( A^\pi_{s,a} - A^*_{s,a} \right)$ in Lemma~\ref{lemma: condition of finite time of PPG and PQD}, one has 
\begin{align*}
 g^\pi_{s,a}-g^\pi_{s,a'} & = \left(\eta_s A^*_{s,a}+\varepsilon^\pi_{s,a}-\pi_{s,a}^++\pi_{s,a}\right)-\left(\eta_s A^*_{s,a'}+\varepsilon^\pi_{s,a'}-\pi_{s,a'}^++\pi_{s,a'}\right)\\
 &=\eta_s\left(A^*_{s,a}-A^*_{s,a'}\right)+(\varepsilon^\pi_{s,a}-\varepsilon^\pi_{s,a'})-(\pi^+_{s,a}-\pi_{s,a})+(\pi^+_{s,a'}-\pi_{s,a'})\\
 &\geq \eta_s\Delta- 2\|\varepsilon^\pi_s\|_\infty-\|\pi^+_s-\pi_s\|_\infty-b_s^\pi.\numberthis\label{eq:ke-tmp-01}
\end{align*}
In addition, setting $p'=\pi_s$ in \eqref{eq:opt-cond-001} yields 
\begin{align*}
\|\pi_s^+-\pi_s\|_2^2&\leq \eta_s\sum_a\pi_{s,a}^+A^\pi_{s,a}\leq \eta_s \sum_{s'}\sum_a\pi_{s',a}^+A^\pi_{s',a} 
% &=\eta_s \sum_{s'}\frac{d_\mu^\pi(s')}{d_\mu^\pi(s')}\sum_a\pi_{s',a}^+A^\pi_{s',a} 
= \eta_s \sum_{s'}\frac{d_\mu^{\pi^+}(s')}{d_\mu^{\pi^+}(s')}\sum_a\pi_{s',a}^+A^\pi_{s',a} \\
&\leq \frac{\eta_s}{(1-\gamma)\tilde{\mu}} (V^{\pi^+}(\mu)-V^\pi(\mu))\leq \frac{\eta_s}{(1-\gamma)\tilde{\mu}} (V^{*}(\mu)-V^\pi(\mu)). 
\end{align*}
Together with \eqref{eq:ke-tmp-01}, one has $g^\pi_{s,a}-g^\pi_{s,a'}>0$ provided 
\begin{align*}
    b_{s}^{\pi}+2\|\varepsilon _{s}^{\pi}\|_{\infty}+\min\left\{\sqrt{\frac{\eta_s}{(1-\gamma)\tilde{\mu}} (V^{*}(\mu)-V^\pi(\mu))}, \;1\right\}<\eta_s\Delta .
\end{align*}
It is clear that this condition (but not as concise as the one presented in Lemma~\ref{lemma: condition of finite time of PPG and PQD}) can also be used to derive the finite iteration convergence of PPG and PQA for any constant step size.
%%%%
\subsubsection{Short discussion II}
In \cite{Li_Zhao_Lan_2022}, the finite iteration convergence of  homotopic policy mirror ascent methods under certain Bregman divergence is investigated. When considering the case where the Bregman divergence is generated by  the squared Euclidean distance, it reduces to the following homotopic PQA method: 
\begin{align*}
\pi_s^{k+1} &= \argmax_{p\in\Delta}\;\eta_k\left[\langle Q^{k}(s,\cdot),p\rangle - \frac{\tau_k}{2}\|p-\pi_s^0\|_2^2\right]-\frac{1}{2}\|p-\pi_s^k\|_2^2\\
&=\argmin_{p\in\Delta}\; \frac{1}{2}\left\|p-\frac{1}{1+\eta_k\tau_k}\pi_s^k-\frac{\eta_k}{1+\eta_k\tau_k}Q^k(s,\cdot)\right\|_2^2\\
&=\mathrm{Proj}_\Delta\left(\frac{1}{1+\eta_k\tau_k}\pi_s^k+\frac{\eta_k}{1+\eta_k\tau_k}Q^k(s,\cdot)\right)\\
&=\mathrm{Proj}_\Delta\left(\frac{1}{1+\eta_k\tau_k}\pi_s^k+\frac{\eta_k}{1+\eta_k\tau_k}A^k(s,\cdot)\right),
\end{align*}
where $\pi_s^0$ is a uniform policy, $\tau_k$ is the regularization parameter, and the last line follows from the fact that $\frac{\eta_k}{1+\eta_k\tau_k}V^k(s)\cdot 1$ is a vector with all the same entries.  It follows that there exists $\lambda^k_s$ such that\footnote {Note that in \cite{Li_Zhao_Lan_2022}, a slightly different version is indeed considered. That is,  if $\pi^k_{s,a}=0$, the starting point can be negative due to the requirement for the careful selection of the subgradient in order to  establish the finite iteration convergence  of the algorithm for exponentially increasing step sizes. }
\begin{align*}
\pi_{s,a}^{k+1}=\frac{1}{1+\eta_k\tau_k}\left(\pi_{s,a}^k+\eta_kA^k(s,a)-\lambda^k_s\right)_+\quad\mbox{and}\quad\sum_a\pi^{k+1}_{s,a}=1.
\end{align*}
Consider  the case where $\eta_k\tau_k$ is fixed, for example  $1+\eta_k\tau_k=1/\gamma$ with $0<\gamma<1$ as considered in \cite{Li_Zhao_Lan_2022}. Then the update reduces to 
    \begin{align*}
\pi_{s,a}^{k+1}=\frac{1}{1/\gamma}\left(\pi^k_{s,a}+\eta_k A^k(s,a)-\lambda_s^k\right)_+\quad\mbox{and}\quad\sum_a\left(\pi^k_{s,a}+\eta_k A^k(s,a)-\lambda_s^k\right)_+=\frac{1}{\gamma}.\numberthis\label{eq:ke-tmp-02}
\end{align*}
Note that this update is overall similar to the update of PQA, differing only in the extra factor $\frac{1}{1/\gamma}$. { However, next we will use a very simple example to show that it requires $\eta_k$ to be sufficiently large for \eqref{eq:ke-tmp-02} to be able to convergence. Therefore, even the finite iteration convergence of \eqref{eq:ke-tmp-02} holds, it does not leads to the finite iteration convergence of PQA for any constant step size.}  
 
 More precisely, 
consider the bandit case where there are only two actions $a_1$ and $a_2$. Assume $a_1$ is the single optimal action. Suppose $\pi^k$ is already optimal, i.e., $\pi_{a_1}^k=1$ and $\pi^k_{a_2}=0$.
% \begin{align*}
% \pi^k_{a_1} = 1\quad\mbox{and}\quad\pi^k_{a_2}=0.
% \end{align*}
Then $A^k_{a_1}=0$ and $A^k_{a_2}<0$. Letting $\Delta=|A^k_{a_2}|$, there exists a $\lambda^k$ such that 
  \begin{align*}
  \pi_{a_1}^{k+1}=\gamma(1-\lambda^k)_+, \quad
  \pi^{k+1}_{a_2}=\gamma(-\eta_k\Delta_k-\lambda^k)_+.
  \end{align*}
      Moreover, 
\begin{align*}
(1-\lambda^k)_++(-\eta_k\Delta-\lambda^k)_+=\frac{1}{\gamma}>1.\numberthis\label{eq:ke-tmp-03}
\end{align*}
 First note that there must hold $\lambda^k<0$; otherwise the above equality cannot hold since $\eta_k\Delta>0$.     Assume $\eta_k\Delta<\frac{1}{\gamma}-1$. Then it is easy to verify by contradiction that one should have  $-\lambda^k>\eta_k\Delta$ in order to satisfy \eqref{eq:ke-tmp-03}. It follows that 
\begin{align*}
    \lambda^k = \frac{1}{2}(1-1/\gamma-\eta_k\Delta)>1-\frac{1}{\gamma}.
\end{align*}
 Therefore, when $\eta_k\Delta<\frac{1}{\gamma}-1$, one has $\pi_{a_1}^{k+1}=\gamma(1-\lambda^k)_+<1$.
% \begin{align*}
% \pi_{a_1}^{k+1}=\gamma(1-\lambda^k)_+<1.
% \end{align*}
  That is, $\pi^{k+1}$ is not optimal anymore. In other words, in order for $\pi^{k+1}$ still to be optimal, one must have $\eta_k\Delta \geq 1/\gamma-1$, that is, $\eta_k\geq (1/\gamma-1)/\Delta$ which can very large when $\Delta$ is small.

%%%%
}
%%%%%%%%%

\subsection{Finite iteration convergence of PI and VI}
As a by-product, we will derive a new dimension-free bound for the finite iteration convergence of policy iteration (PI) and value iteration (VI) in terms of $\Delta$ in this section. The following lemma demonstrates that once a vector is sufficiently close to the optimal value vector, then the policy retrieved from that vector is an optimal policy.%the current optimal action set  $\mathcal{A}^\pi_s$ will be included in the global optimal action set $\mathcal{A}^*_s$.

\begin{lemma}%For any policy $\pi \in \Pi$, if 
%$$
%\gamma \left\| V^*-V^{\pi} \right\| _{\infty}\le \frac{\Delta}{3},
%$$
%then for all $s\in \mathcal{S}$, $\mathcal{A} _{s}^{\pi}\subset\mathcal{A} _{s}^{*}$.
For any $V\in\R^{|\mathcal{S}|}$ (not necessarily associated with a policy), define $Q^V\in\R^{|\mathcal{S}|\times |\mathcal{A}|}$ as follows:
\begin{align*}
Q^V(s,a) = \mathbb{E}_{s'\sim P(\cdot|s,a)}[r(s,a,s')+\gamma V(s')].
\end{align*}
If $\gamma \|V^*-V\|_\infty\leq \frac{\Delta}{3}$, then  $\arg\max\limits_a Q^V(s,)\subset \mathcal{A}^*_s$. That is, the greedy policy supported on $\arg\max\limits_a Q^V(s,a)$ is an optimal policy.
\label{lemma: A^k_s = A^*_s}
\end{lemma}
%\begin{remark} This lemma states that once the value error enter into the small-$\varepsilon$ regime, then the current optimal action set $\mathcal{A}^\pi_s$ is included by the global optimal action set $\mathcal{A}^*_s$.
%\end{remark}

\begin{proof}
    %By lemma \ref{bounds of V Q A} we know that 
    First, it is easy to see that $\forall s,a$, 
    \begin{align*}
       |Q^*(s,a)-Q^V(s,a)| = \gamma|\mathbb{E}_{s'\sim P(\cdot|s,a)}[V^*(s')-V(s')]|\leq \gamma \|V^*-V\|_\infty\leq \frac{\Delta}{3}.
    \end{align*}
It follows that for  $s$ having non-optimal actions, $a\in\mathcal{A} _{s}^{*}$ and $a'\not\in\mathcal{A} _{s}^{*}$, we have
\begin{align*}
Q^V(s,a)&\geq Q^*(s,a)-\frac{\Delta}{3} \geq Q^*(s,a')+\frac{2\Delta}{3}\geq Q^V(s,a')+\frac{\Delta}{3}>Q^V(s,a'),
\end{align*}
which concludes the proof.
\end{proof}

\begin{theorem}[Finite iteration convergence of PI]
    PI terminates after at most 
    \begin{align*}
        k_0=\left\lceil\frac{1}{1-\gamma} \log \left( \frac{3}{(1-\gamma) \Delta} \right)\right\rceil
    \end{align*}
    iterations.
    \label{theorem: upper bound of PI iteration}
\end{theorem}
\begin{proof} Notice that the value error generated by PI satisfies,
$$
\left\| V^*-V^{k} \right\| _{\infty}\le \gamma ^k\left\| V^*-V^{0} \right\| _{\infty}\le \frac{\gamma ^k}{1-\gamma},
$$
see for example \cite{Bertsekas} for the proof.
According to Lemma \ref{lemma: A^k_s = A^*_s}, when 
\begin{align}
\frac{\gamma ^{k+1}}{1-\gamma}\le \frac{\Delta}{3},
\label{termination condition of PI}
\end{align}
we have $\mathcal{A}^k_s \subset \mathcal{A}^*_s$ after that. It's trivial to verify that \eqref{termination condition of PI} holds for $\pi^k$ when $k\geq k_0$. Since PI puts all the probabilities on the action set $\mathcal{A}^k_s$ in each iteration,  we have $\calA_s^k \subset \calA^*_s$ when $k \ge k_0$,
which implies PI outputs an optimal policy after $k_0$.
\end{proof}

\begin{remark} It is well-known that PI is a strong polynomial algorithm \textup{(}see for example \textup{\cite{pi_Bruno}}\textup{)}, which means 
  PI outputs an optimal policy after $
\mathcal{O} \left( \frac{\left| \mathcal{S} \right|\left| \mathcal{A} \right|}{1-\gamma}\log \frac{1}{1-\gamma} \right)$ iterations. Compared with this strong polynomial bound, the bound in Theorem \ref{theorem: upper bound of PI iteration} is dimension-free but relies on the parameter $\Delta$ \kw{that depends on the particular MDP problem}. %Thus the upper bound in Theorem \ref{theorem: upper bound of PI iteration} do not match the strong polynomial property as shown in \cite{pi_Bruno}.
The dimension-free bound is better in the case $
\frac{1}{\Delta}=o\left( \frac{1}{\left( 1-\gamma \right) ^{\left| \mathcal{S} \right|\left| \mathcal{A} \right|}} \right)$.

\end{remark}

%For VI, since this method does not evaluate the policy value function directly, thus we need to bound the policy value error generated by VI firstly. 

\begin{theorem} Let $\pi^k$ be the sequence of greedy policy generated by $V^k$ in VI (Note that $V^k$ is not necessarily a value function of $\pi^k$). Then after at most 
$$
k_0:=\left\lceil\frac{1}{ 1-\gamma }\log \left( \frac{3\|V^*-V^0\|_\infty}{\Delta}\right)\right\rceil
$$
iterations, $\pi^k$  is an optimal policy.
\label{theorem: upper bound of VI}
\end{theorem}
\begin{proof} %According to Lemma \ref{Error amplification}, we have
The value error generated by VI satisfies 
$$
\left\| V^*-V^{{k}} \right\| _{\infty}\le \gamma^k\|V^*-V^0\|_\infty\le \frac{\Delta}{3},
$$
where the second inequality follows from the assumption. Then the application of Lemma~\ref{lemma: A^k_s = A^*_s} concludes the proof.
%Then the proof is  completed following the same argument as in the proof of Theorem \ref{theorem: upper bound of PI iteration}.
\end{proof}
\begin{remark} It is worth noting that since VI does not evaluate  the value function of $\pi^k$ in each iteration,  Theorem \ref{theorem: upper bound of VI} does not really mean the algorithm terminates in a finite number of iterations.
\end{remark}

\section{Linear convergence and equivalence to PI}
\label{Linear convergence and the equivalence to PI}

\subsection{Linear convergence of PPG under non-adaptive increasing step sizes}
In Theorem~\ref{theorem:sublinear of PPG}, we have established the sublinear convergence of PPG for constant step sizes.  In this section, we further show that  with increasing step sizes $\eta_k \ge \calO\left( \frac{1}{\gamma^{2k}} \right)$, the classical $\gamma$-rate linear convergence of PPG can be achieved globally.
Note that this result can indeed be obtained based on a similar argument for PQA in 
\cite{Johnson_Pike-Burke_Rebeschini_2023}. Here,  for the sake of
self-completeness, we present a different proof based on Lemma~\ref{lemma: advantage lower bound for support actions} instead of the three point descent lemma  used in \cite{Johnson_Pike-Burke_Rebeschini_2023}.

\begin{theorem}Consider the prototype update in \eqref{policy update of PPG and PQD}. Suppose the step size in the $k$-th iteration satisfies 
\begin{align}
\eta^k_s \ge \frac{1}{\gamma ^{2k+1}c_0}\cdot 2 \pi _{s}^{k}\left( \mathcal{A} \setminus \mathcal{A} _{s}^{k} \right), \quad \forall s\in \mathcal{S}, 
\label{condition_of_linear_convergece_by_control_error}
\end{align}
for a given constant $c_0 > 0$. Then the value errors satisfy
$$
\left\| V^*-V^k \right\| _{\infty}<\gamma ^k\left( \left\| V^*-V^0 \right\| _{\infty}+\frac{c_0}{1-\gamma} \right).
$$
\label{thm:ketmp02}
\end{theorem}

\begin{proof}

For simplicity of notation, let $
\tilde{\eta}_{s}^{k}:=\frac{2\pi _{s}^{k}\left( \mathcal{A} \setminus \mathcal{A} _{s}^{k} \right)}{\eta _{s}^{k}}$. According to Lemma \ref{lemma: advantage lower bound for support actions}, for any $k > 0$ and $s \in \calS$,
$$
\sum_{a\in \mathcal{A}}{\pi _{s,a}^{k+1}Q_{s,a}^{k}}\ge \sum_{a\in \mathcal{A}}{\pi _{s,a}^{k+1}\left( \underset{\tilde{a}\in \mathcal{A}}{\max}\,Q_{s,\tilde{a}}^{k}-\tilde{\eta}_{s}^{k} \right)}=\underset{\tilde{a}\in \mathcal{A}}{\max}\,Q_{s,\tilde{a}}^{k}-\tilde{\eta}_{s}^{k}.
$$
Then
\begin{align*}
&\phantom{==}V^*\left( s \right) -V^{k+1}\left( s \right) =V^*\left( s \right) -\mathbb{E} _{a\sim \pi _{s}^{k+1}}\left[ Q^{k+1}_{s,a} \right] 
\le V^*\left( s \right) -\mathbb{E} _{a\sim \pi _{s}^{k+1}}\left[ Q^k_{s,a} \right] 
\\
\,\,                    &\le V^*\left( s \right) -\left( \underset{\tilde{a}\in \mathcal{A}}{\max}\,Q_{s,\tilde{a}}^{k}-\tilde{\eta}^k_s \right) 
=\underset{a\in \mathcal{A}}{\max}\,Q^*_{s,a} -\underset{\tilde{a}\in \mathcal{A}}{\max}\,Q_{s,\tilde{a}}^{k} + \tilde{\eta}^k_s
\le \gamma \left\| V^*-V^k \right\| _{\infty}+ \tilde{\eta}^k_s,
\end{align*}
where in the first inequality we have used the fact $Q^{k+1}_{s,a}\leq Q^k_{s,a}$ due to the improvement. 
It follows that
\begin{align*}
\left\| V^*-V^k \right\| _{\infty}&\le \gamma \left\| V^*-V^{k-1} \right\| _{\infty}+\underset{s\in \mathcal{S}}{\max}\,\tilde{\eta}_{s}^{k}
\\
&\,\,\le \gamma ^2\left\| V^*-V^{k-2} \right\| _{\infty}+\gamma \underset{s\in \mathcal{S}}{\max}\,\tilde{\eta}_{s}^{k-1}+\underset{s\in \mathcal{S}}{\max}\,\tilde{\eta}_{s}^{k}\leq ...
\\
&\,\,\le 
\gamma ^k\left\| V^*-V^0 \right\| _{\infty}+\sum_{i=0}^{k-1}{\left( \underset{s\in \mathcal{S}}{\max}\,\tilde{\eta}_{s}^{i} \right) \gamma ^{k-1-i}}.
\numberthis
\label{eqC}
\end{align*}
Notice that the condition \eqref{condition_of_linear_convergece_by_control_error} is equivalent to 
$
\underset{s\in \mathcal{S}}{\max}\,\tilde{\eta}_{s}^{i}\le c_0\gamma ^{2i+1}.
$
Plugging it into \eqref{eqC} yields
$$
\left\| V^*-V^k \right\| _{\infty}\le \gamma ^k\left\| V^*-V^0 \right\| _{\infty}+c_0\sum_{i=0}^{k-1}{\gamma ^{2i+1}\gamma ^{k-1-i}}<\gamma ^k\left( \left\| V^*-V^0 \right\| _{\infty}+\frac{c_0}{1-\gamma} \right),
$$
which completes the proof.
\end{proof}

The $\gamma$-rate linear convergence of PPG follows immediately by noting that $\eta_s^k=\eta_k\frac{d_\mu^k(s)}{1-\gamma}$ in PPG and $d_\mu^k(s)\geq (1-\gamma)\tilde{\mu}$.
\begin{proposition} For PPG, if
$
\eta _k\ge \frac{1}{\tilde{\mu}}\frac{1}{c_0}\frac{2}{\gamma ^{2k+1}},
$
then 
\begin{align}
\left\| V^*-V^k \right\| _{\infty}<\gamma ^k\left( \left\| V^*-V^0 \right\| _{\infty}+\frac{c_0}{1-\gamma} \right).
\label{non-adaptive PPG recover gamma-rate}
\end{align}
\label{proposition: non-adaptive PPG recover gamma-rate}
\end{proposition}
%the condition \eqref{condition_of_linear_convergece_by_control_error} is satisfied if 

%\input{theorems/section 5/corollary: linear convergence of PQD}

\begin{remark} Recalling from Lemma~\ref{corollary: condition of finite time of PPG and PQD w.r.t value error} that when the value error satisfies
\begin{align}
\left\| V^*-V^k \right\| _{\infty}\le \frac{\Delta}{2}\frac{\eta _{s}^{k}\Delta}{1+\eta _{s}^{k}\Delta},
\label{convergence threhold of protopy update}
\end{align}
the prototype update in  \eqref{policy update of PPG and PQD} outputs an optimal policy. Using the step sizes in Proposition \ref{proposition: non-adaptive PPG recover gamma-rate} for PPG, it is easy to see that the RHS of \eqref{convergence threhold of protopy update} satisfies
\begin{align}
\frac{\Delta}{2}\frac{\eta _{s}^{k}\Delta}{1+\eta _{s}^{k}\Delta}= \frac{\Delta}{2}\left(1- \frac{1}{1+\eta^k _s\Delta} \right) \ge \frac{\Delta}{2} \left( 1 - \frac{1}{1+\eta_k\tilde{\mu}\Delta} \right) 
\ge \frac{\Delta}{2}\left( \frac{2}{c_0+2} \right).
\label{RHS lb}
\end{align}
Combining \eqref{non-adaptive PPG recover gamma-rate}, \eqref{convergence threhold of protopy update} and \eqref{RHS lb} together implies that,  after at most
$$
k_0:=\left\lceil\frac{1}{1-\gamma}\log \left( \frac{(c_0+1)(c_0+2)}{(1-\gamma)\Delta} \right) \right\rceil
$$
 iterations, PPG with the non-adaptive increasing step sizes achieves exact convergence.
\end{remark}

\subsection{Equivalence of PPG  to PI under adaptive step sizes}
\label{Finite time Convergence of PPG and PQD with adaptive step size}
As already mentioned, 
it is easy to see PPG should converge to a PI update when $\eta_s\rightarrow\infty$.
In this section, we study the convergence  of PPG with adaptive step sizes and identify the non-asymptotic step size threshold beyond which  PPG is equivalent to PI. The analysis of this section is similar to that for the finite iteration convergence. We utilize the gap property (Lemma \ref{proj_property_2}) again to show that once the step size is large enough, then the action set $\mathcal{A} \setminus \mathcal{A}^k_s$ will be eliminated from the support set of the new policy. 

\begin{theorem} Consider the prototype update in \eqref{policy update of PPG and PQD} and suppose the step size $\eta$ satisfies
\begin{align}
\underset{s\in \mathcal{S}}{\min}\,\,\eta _s > \mathcal{F} ^{\pi}:=\frac{2}{\Delta ^{\pi}}\cdot \underset{s\in \mathcal{S}}{\max}\left\{ \pi _s\left( \mathcal{A} \setminus \mathcal{A} _{s}^{\pi} \right) \right\},
\label{condition of the equivalent to PI}
\end{align}
where $
\Delta ^{\pi}:=\underset{s\in \mathcal{S}}{\min}\left| \underset{a^{\prime}\in \mathcal{A}}{\max}A_{s,a^{\prime}}^{\pi}-\underset{a^{\prime}\notin \mathcal{A} _{s}^{\pi}}{\max}A_{s,a^{\prime}}^{\pi} \right|$. Then the new policy at state $s$ (i.e., ${\pi}_s^+$) is supported on the action set $\mathcal{A}^\pi_s$, which implies that the prototype update is equivalent to  PI.
\label{theorem: PPG and PQD are PI}
\end{theorem}
\begin{proof} Notice that for each state $s \in \calS$ and $a \notin \mathcal{A}^\pi_s$, $
A_{s,a}^{\pi}\le \underset{a^{\prime}\notin \mathcal{A} _{s}^{\pi}}{\max}A_{s,a^{\prime}}^{\pi}$. By Lemma \ref{lemma: advantage lower bound for support actions}, when the step size satisfies
$$
\frac{2\pi _s\left( \mathcal{A} \setminus \mathcal{A} _{s}^{\pi} \right)}{\eta _s}<\underset{a^{\prime}\in \mathcal{A}}{\max}A_{s,a^{\prime}}^{\pi}-\underset{a^{\prime}\notin \mathcal{A} _{s}^{\pi}}{\max}A_{s,a^{\prime}}^{\pi},
$$
or equivalently
\begin{align}
\eta _s<\frac{2\pi _s\left( \mathcal{A} \setminus \mathcal{A} _{s}^{\pi} \right)}{\underset{a^{\prime}\in \mathcal{A}}{\max}A_{s,a^{\prime}}^{\pi}-\underset{a^{\prime}\notin \mathcal{A} _{s}^{\pi}}{\max}A_{s,a^{\prime}}^{\pi}},
\label{PI_condition_for_each_s}
\end{align}
all the $ a^\prime \not\in\calA^\pi_s$ are not in $\calB_s(\eta_s)$. That is, the new policy $\pi^+_s$ is supported on $\mathcal{A}^\pi_s$. It's trivial to see that the condition \eqref{condition of the equivalent to PI} implies \eqref{PI_condition_for_each_s} for every $s \in \mathcal{S}$, thus the proof is completed.
\end{proof}

The equivalence of PPG to PI follows immediately from this theorem, which is similarly applicable for PQA.

\begin{corollary}If the step size satisfies $
\eta_k \ge \frac{1}{\tilde{\mu}}\mathcal{F} ^{\pi^k},$
PPG is equivalent to PI.
\label{corollary: PPG is PI}
\end{corollary}

\begin{corollary}If the step size satisfies $\eta_k \ge \mathcal{F} ^{\pi^k},$
 PQA is equivalent to PI.
\end{corollary}

\begin{remark}
It is worth noting that the step size threshold in the above two corollaries  only relies on the current policy $\pi^k$.
\end{remark}

\bibliographystyle{plain}
\bibliography{refs}

\begin{thebibliography}{10}

\bibitem{Agarwal2016MakingCD}
Alekh Agarwal, Sarah Bird, Markus Cozowicz, Luong Hoang, John Langford, Stephen
  Lee, Jiaji Li, Daniel~R. Melamed, Gal Oshri, Oswaldo Ribas, Siddhartha Sen,
  and Alex Slivkins.
\newblock Making contextual decisions with low technical debt.
\newblock {\em arXiv:1606.03966}, 2016.

\bibitem{Agarwal_Kakade_Lee_Mahajan_2019}
Alekh Agarwal, Sham~M. Kakade, Jason~D. Lee, and Gaurav Mahajan.
\newblock On the theory of policy gradient methods: {O}ptimality,
  approximation, and distribution shift.
\newblock {\em Journal of Machine Learning Research}, 22(98):1--76, 2021.

\bibitem{Beck}
Amir Beck.
\newblock {\em First-{O}rder {M}ethods in {O}ptimization}.
\newblock SIAM and MOR, 2017.

\bibitem{Berner2019Dota2W}
Christopher Berner, Greg Brockman, Brooke Chan, Vicki Cheung, Przemyslaw
  Debiak, Christy Dennison, David Farhi, Quirin Fischer, Shariq Hashme,
  Christopher Hesse, Rafal Jozefowicz, Scott Gray, Catherine Olsson, Jakub
  Pachocki, Michael Petrov, HenriquePondedeOliveira Pinto, Jonathan Raiman, Tim
  Salimans, Jeremy Schlatter, Jonas Schneider, Szymon Sidor, Ilya Sutskever,
  Jie Tang, Filip Wolski, and Susan Zhang.
\newblock Dota 2 with large scale deep reinforcement learning.
\newblock {\em arXiv:1912.06680}, 2019.

\bibitem{Bertsekas}
Dimitri~P. Bertsekas.
\newblock {\em Reinforcement {L}earning and {O}ptimal {C}ontrol}.
\newblock Athena Scientific, 2019.

\bibitem{bhabdari2024or}
Jalj Bhandari and Daniel Russo.
\newblock Global optimality guarantees for policy gradient methods.
\newblock {\em Operations Research}, 2024.

\bibitem{Cen_Cheng_Chen_Wei_Chi_2022}
Shicong Cen, Chen Cheng, Yuxin Chen, Yuting Wei, and Yuejie Chi.
\newblock Fast global convergence of natural policy gradient methods with
  entropy regularization.
\newblock {\em Operations Research}, 70(4):2563--2578, 2022.

\bibitem{chen-2019-topk}
Minmin Chen, Alex Beutel, Paul Covington, Sagar Jain, Francois Belletti, and
  Ed~H. Chi.
\newblock Top-k off-policy correction for a {REINFORCE} recommender system.
\newblock In {\em Proceedings of the Twelfth ACM International Conference on
  Web Search and Data Mining}, pages 456--464, 2019.

\bibitem{Geist_Scherrer_Pietquin_2019}
Matthieu Geist, Bruno Scherrer, and Olivier Pietquin.
\newblock A theory of regularized markov decision processes.
\newblock In {\em Proceedings of the 36th International Conference on Machine
  Learning}, pages 2160--2169, 2019.

\bibitem{robot1}
Jemin Hwangbo, Joonho Lee, Alexey Dosovitskiy, Dario Bellicoso, Vassilios
  Tsounis, Vladlen Koltun, and Marco Hutter.
\newblock Learning agile and dynamic motor skills for legged robots.
\newblock {\em Science Robotics}, 4(26), 2019.

\bibitem{Johnson_Pike-Burke_Rebeschini_2023}
Emmeran Johnson, Ciara Pike-Burke, and Patrick Rebeschini.
\newblock Optimal convergence rate for exact policy mirror descent in
  discounted markov decision processes.
\newblock In {\em Sixteenth European Workshop on Reinforcement Learning}, 2023.

\bibitem{kakade2002approximately}
Sham~M. Kakade and John Langford.
\newblock Approximately optimal approximate reinforcement learning.
\newblock In {\em International Conference on Machine Learning (ICML '02)},
  pages 267--274, 2002.

\bibitem{Khodadadian_Jhunjhunwala_Varma_Maguluri_2021}
Sajad Khodadadian, Prakirt~Raj Jhunjhunwala, Sushil~Mahavir Varma, and
  Siva~Theja Maguluri.
\newblock On the linear convergence of natural policy gradient algorithm.
\newblock In {\em 2021 60th IEEE Conference on Decision and Control (CDC)},
  pages 3794--3799, 2021.

\bibitem{Lan_2021}
Guanghui Lan.
\newblock Policy mirror descent for reinforcement learning: {L}inear
  convergence, new sampling complexity, and generalized problem classes.
\newblock {\em Mathematical Programming}, 198(1):1059--1106, 2021.

\bibitem{robot2}
Joonho Lee, Jemin Hwangbo, Lorenz Wellhausen, Vladlen Koltun, and Marco Hutter.
\newblock Learning quadrupedal locomotion over challenging terrain.
\newblock {\em Science Robotics}, 5(47), 2020.

\bibitem{Li_Zhao_Lan_2022}
Yan Li, Guanghui Lan, and Tuo Zhao.
\newblock Homotopic policy mirror descent: Policy convergence, implicit
  regularization, and improved sample complexity.
\newblock {\em Mathematical Programming}, 2023.

\bibitem{hadamard_pg}
Jiacai Liu, Jinchi Chen, and Ke~Wei.
\newblock On the linear convergence of policy gradient under {H}adamard
  parameterization.
\newblock {\em arXiv:2305.19575}, 2023.

\bibitem{Mei_Xiao_Szepesvari_Schuurmans_2020}
Jincheng Mei, Chenjun Xiao, Csaba Szepesvári, and Dale Schuurmans.
\newblock On the global convergence rates of softmax policy gradient methods.
\newblock In {\em Proceedings of the 37th International Conference on Machine
  Learning}, volume 119, pages 6820--6829, 2020.

\bibitem{robot3}
Takahiro Miki, Joonho Lee, Jemin Hwangbo, Lorenz Wellhausen, Vladlen Koltun,
  and Marco Hutter.
\newblock Learning robust perceptive locomotion for quadrupedal robots in the
  wild.
\newblock {\em Science Robotics}, 7(62), 2022.

\bibitem{chipdesign}
Azalia Mirhoseini, Anna Goldie, Mustafa Yazgan, Joe~Wenjie Jiang, Ebrahim
  Songhori, Shen Wang, Young-Joon Lee, Eric Johnson, Omkar Pathak, Azade Nazi,
  Jiwoo Pak, Andy Tong, Kavya Srinivasa, William Hang, Emre Tuncer, Quoc~V. Le,
  James Laudon, Richard Ho, Roger Carpenter, and Jeff Dean.
\newblock A graph placement methodology for fast chip design.
\newblock {\em Nature}, 594(7862):207--212, 2021.

\bibitem{mnih2015humanlevel}
Volodymyr Mnih, Koray Kavukcuoglu, David Silver, Andrei~A Rusu, Joel Veness,
  Marc~G Bellemare, Alex Graves, Martin Riedmiller, Andreas~K Fidjeland, Georg
  Ostrovski, Stig Petersen, Charles Beattie, Amir Sadik, Ioannis Antonoglou,
  Helen King, Dharshan Kumaran, Daan Wierstra, Shane Legg, and Demis Hassabis.
\newblock Human-level control through deep reinforcement learning.
\newblock {\em Nature}, 518(7540):529--533, 2015.

\bibitem{Rockafellar}
R.~Tyrrell Rockafellar.
\newblock {\em Convex {A}nalysis}.
\newblock Princeton University Press, 1996.

\bibitem{pi_Bruno}
Bruno Scherrer.
\newblock Improved and generalized upper bounds on the complexity of policy
  iteration.
\newblock In {\em Advances in Neural Information Processing Systems 26 (NeurIPS
  2013)}, volume~26, 2013.

\bibitem{Shani_Efroni_Mannor_2019}
Lior Shani, Yonathan Efroni, and Shie Mannor.
\newblock Adaptive trust region policy optimization: {G}lobal convergence and
  faster rates for regularized {MDP}s.
\newblock In {\em Proceedings of the AAAI Conference on Artificial
  Intelligence}, pages 5668--5675, 2020.

\bibitem{SilverHuangEtAl16nature}
David Silver, Aja Huang, Chris~J Maddison, Arthur Guez, Laurent Sifre, George
  Van Den~Driessche, Julian Schrittwieser, Ioannis Antonoglou, Veda
  Panneershelvam, Marc Lanctot, Sander Dieleman, Dominik Grewe, John Nham, Nal
  Kalchbrenner, Ilya Sutskever, Timothy Lillicrap, Madeleine Leach, Koray
  Kavukcuoglu, Thore Graepel, and Demis Hassabis.
\newblock Mastering the game of {G}o with deep neural networks and tree search.
\newblock {\em Nature}, 529(7587):484--489, 2016.

\bibitem{pg}
Richard~S Sutton, David McAllester, Satinder Singh, and Yishay Mansour.
\newblock Policy gradient methods for reinforcement learning with function
  approximation.
\newblock In {\em Proceedings of the 12th International Conference on Neural
  Information Processing Systems}, pages 1057--1063, 1999.

\bibitem{StarCraft}
Oriol Vinyals, Timo Ewalds, Sergey Bartunov, Petko Georgiev, Alexander
  Vezhnevets, Michelle Yeo, Alireza Makhzani, Heinrich Küttler, JohnP.
  Agapiou, Julian Schrittwieser, John Quan, Stephen Gaffney, Stig Petersen,
  Karen Simonyan, Tom Schaul, Hadovan Hasselt, David Silver, TimothyP.
  Lillicrap, Kevin Calderone, Paul Keet, Anthony Brunasso, David Lawrence,
  Anders Ekermo, Jacob Repp, and Rodney Tsing.
\newblock Starcraft {II}: {A} new challenge for reinforcement learning.
\newblock {\em arXiv:1708.04782}, 2017.

\bibitem{simplex_projection}
Weiran Wang and Miguel~A. Carreira-Perpiñán.
\newblock Projection onto the probability simplex: {A}n efficient algorithm
  with a simple proof, and an application.
\newblock {\em arXiv:1309.1541}, 2013.

\bibitem{Xiao_2022}
Lin Xiao.
\newblock On the convergence rates of policy gradient methods.
\newblock {\em Journal of Machine Learning Research}, 23(282):1--36, 2022.

\bibitem{Zhan_Cen_Huang_Chen_Lee_Chi_2021}
Wenhao Zhan, Shicong Cen, Baihe Huang, Yuxin Chen, Jason~D. Lee, and Yuejie
  Chi.
\newblock Policy mirror descent for regularized reinforcement learning: {A}
  generalized framework with linear convergence.
\newblock {\em SIAM Journal on Optimization}, 33(2):1061--1091, 2023.

\bibitem{Zhang_Koppel_Bedi_Szepesvari_Wang_2020}
Junyu Zhang, Alec Koppel, Amrit~Singh Bedi, Csaba Szepesvári, and Mengdi Wang.
\newblock {V}ariational policy gradient method for reinforcement learning with
  general utilities.
\newblock In {\em Advances in Neural Information Processing Systems 33 (NeurIPS
  2020)}, volume~33, pages 4572--4583, 2020.

\end{thebibliography}

\end{document}